\UseRawInputEncoding
\RequirePackage{fix-cm}
\documentclass[smallextended]{svjour3}       
\smartqed  
\usepackage{color}
\usepackage[colorlinks=true,
            citecolor=blue,
            linkcolor=blue,
            anchorcolor=green,
            urlcolor=blue
            ]{hyperref}
\usepackage{graphicx}
\usepackage{lipsum}
\usepackage{amsfonts}
\usepackage{graphicx}
\usepackage{epstopdf}
\usepackage{algorithm}
\usepackage{algorithmic}
\usepackage{amsmath}
\usepackage{amssymb}
\usepackage{bm}
\usepackage{mathbbold}
\usepackage{mathrsfs}
\usepackage{multirow}
\usepackage{dcolumn}
\usepackage{listings}
\usepackage{subfigure}
\usepackage{enumerate}
\usepackage{marvosym}
\numberwithin{theorem}{section}
\numberwithin{corollary}{section}
\numberwithin{lemma}{section}
\numberwithin{remark}{section}
\numberwithin{definition}{section}

\newcommand{\ud}{\mathrm{d}}

\newtheorem{assum}{Assumption~}
\usepackage{amsopn}

\usepackage{mathptmx}      
\usepackage{latexsym}
\journalname{Mathematical Programming}
\begin{document}

\title{Asymptotic proximal point methods: finding the global minima with linear convergence for a class of multiple minima problems}
\titlerunning{Asymptotic proximal point methods}        

\author{Xiaopeng Luo \and Xin Xu \and Herschel A. Rabitz}


\institute{\Letter~Xiaopeng Luo \\
\email{luo.permanent@gmail.com} \\
\at Department of Control and Systems Engineering, School of Management and Engineering, Nanjing University, Nanjing, 210008, China \\
Department of Chemistry, Princeton University, Princeton, NJ 08544, USA
\and
Xin Xu \\
\email{xu.permanent@gmail.com} \\
\at Department of Control and Systems Engineering, School of Management and Engineering, Nanjing University, Nanjing, 210008, China \\
Department of Chemistry, Princeton University, Princeton, NJ 08544, USA
\and
Herschel A. Rabitz \\
\email{hrabitz@princeton.edu} \\
\at Department of Chemistry, Princeton University, Princeton, NJ 08544, USA \\
Program in Applied and Computational Mathematics, Princeton University, Princeton, NJ 08544, USA
}

\date{Received: date / Accepted: date}

\maketitle

\begin{abstract}
  We propose and analyze asymptotic proximal point (APP) methods to find the global minimizer for a class of nonconvex, nonsmooth, or even discontinuous multiple minima functions. The method is based on an asymptotic representation of nonconvex proximal points so that it can find the global minimizer without being trapped in saddle points, local minima, or even discontinuities. Our main result shows that the method enjoys the global linear convergence for such a class of functions. Furthermore, the method is derivative-free and its per-iteration cost, i.e., the number of function evaluations, is also bounded, so it has a complexity bound $\mathcal{O}(\log\frac{1}{\epsilon})$ for finding a point such that the gap between this point and the global minimizer is less than $\epsilon>0$. Numerical experiments and comparisons in various dimensions from $2$ to $500$ demonstrate the benefits of the method.

\end{abstract}

\section{Introduction}
\label{APP:s1}

In this paper, we propose and analyze asymptotic proximal point (APP) methods for finding the global minima
\begin{equation}\label{APP:eq:P}
  x_*=\arg\min_{x\in\mathbb{R}^d}f(x),
\end{equation}
where the objective function $f$ satisfies the following assumption:

\begin{assum}\label{APP:ass:A}
The objective function $f:\mathbb{R}^d\to\mathbb{R}$ satisfies that there exist $x_*\in\mathbb{R}^d$ and $0<l\leqslant L<\infty$ such that for all $x\in\mathbb{R}^d$,
\begin{equation}\label{APP:eq:A}
  f_*+\frac{l}{2}\|x-x_*\|_2^2\leqslant f(x)
  \leqslant f_*+\frac{L}{2}\|x-x_*\|_2^2.
\end{equation}
Hence, $f$ has a unique global minimizer $x_*$ with $f_*:=f(x_*)$.
\end{assum}

\begin{figure}[!h]
\centering
\includegraphics[width=0.4\textwidth]{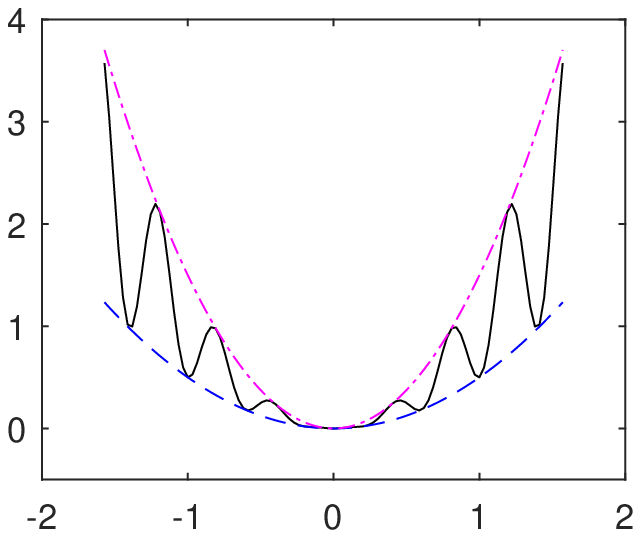}
\includegraphics[width=0.4\textwidth]{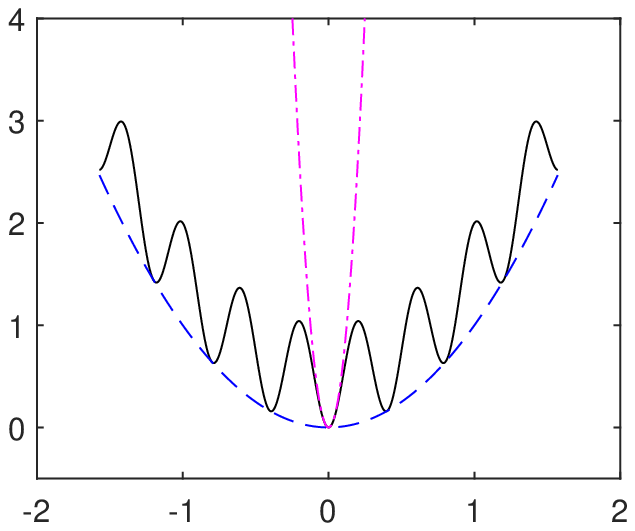}
\caption{One-dimensional examples. Left: the objective $f(x)=x^2+x^2\cos(5\pi x)/2$ (solid) with lower bound $x^2/2$ (dashed) and upper bound $3x^2/2$ (dash-dotted). Right: the objective $f(x)=x^2-\cos(5\pi x)/2+1/2$ (solid) with lower bound $x^2$ (dashed) and upper bound $65x^2$ (dash-dotted).}
\label{APP:fig:A}
\end{figure}

Obviously, such a class of functions is extended from strongly convex functions with Lipschitz-continuous gradients; however, as shown in Figure \ref{APP:fig:A}, it is not ruling out the possibility of local minima. The lower bound $f_*+\frac{l}{2}\|x-x_*\|_2^2$ guarantees the uniqueness of the global minima while the upper bound $f_*+\frac{L}{2}\|x-x_*\|_2^2$ controls the sharpness of the minima. Therefore, the objective $f$ is continuously differentiable at $x_*$ but may be nonsmooth or even discontinuous elsewhere.

\subsection{Motivational problem}

The motivation for introducing Assumption \ref{APP:ass:A} comes from, but is not limited to, the protein folding problem. One of its challenges is to devise an algorithm to accurately predict a protein's native structure from its amino acid sequence, and computational folding explores the process by which the protein proceeds through conformational states to states of lower free energies \cite{DillK2012_PF}. A number of studies \cite{BJ1995_PFFunnel,DillK1985_PF,BryngelsonJ1987_BF,YanZ2020_PFFunnel} indicated that protein-folding energy landscapes are funnel-shaped on a global scale with huge local minima on the local scales (left plot of Figure \ref{APP:fig:PF}).

\begin{figure}[!h]
\centering
\includegraphics[width=0.3\textwidth]{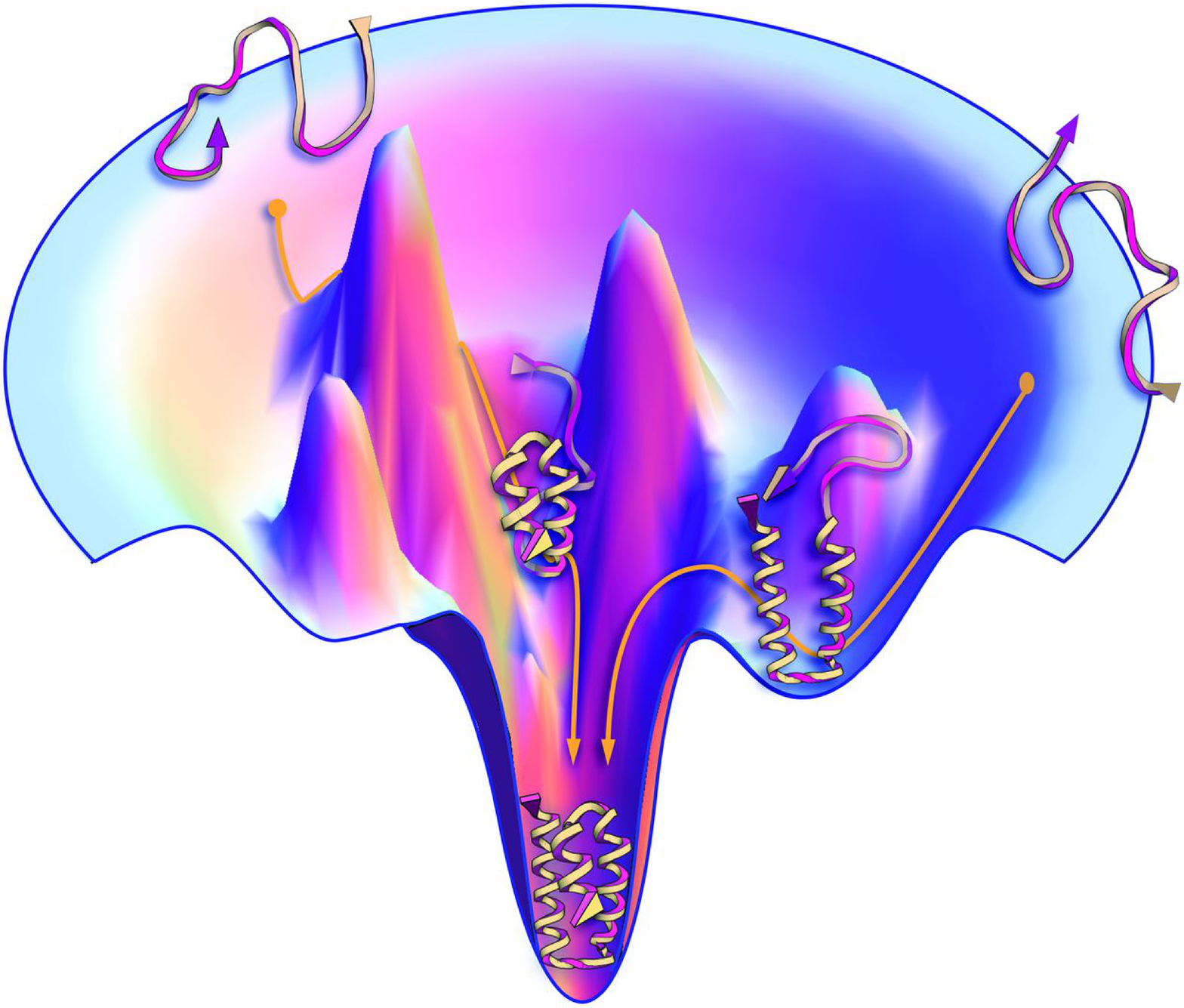}\qquad\quad
\includegraphics[width=0.4\textwidth]{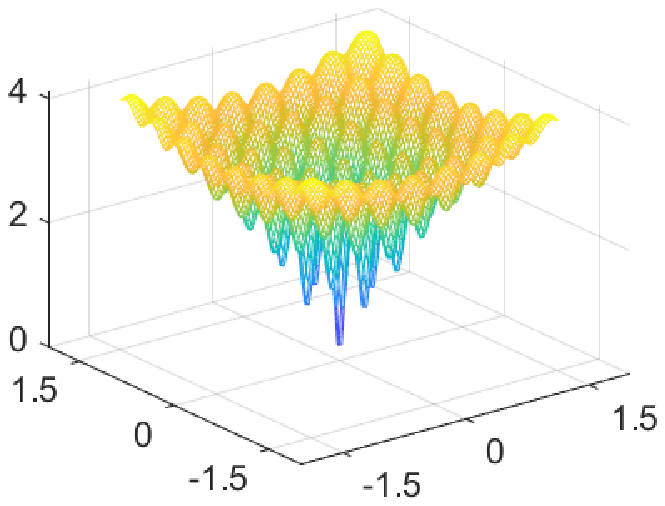}
\caption{Left: proteins have a funnel-shaped energy landscape with many
high-energy, unfolded structures and only a few low-energy, folded structures (Reproduced with the permission from \cite{DillK2012_PF}). Right: the $2$-D revised Rastrigin function defined in Sect. \ref{APP:s4}.}
\label{APP:fig:PF}
\end{figure}

This crucial finding strongly illustrates that the protein-folding energy landscapes is essentially similar to the revised Rastrigin function (right plot of Figure \ref{APP:fig:PF}) which satisfies our Assumption \ref{APP:ass:A} within an appropriate range. Therefore, in the following we will limit our discussion to the minimization of such explicit functions to avoid getting involved in complicated details about biochemistry and molecular dynamics (MD). Obviously, for obtaining the folding pathway from a certain initial state to the thermodynamically stable state \cite{EnglanderS2014_Pathways,WolynesP1995_Pathways}, one needs a sufficiently accurate forcefield model to participate in MD simulations in a suitable bath of explicit water molecules, in addition to using a guaranteed optimization algorithm to avoid the huge number of conformations corresponding to local minima of the folding energy landscape \cite{DillK2012_PF}.

Recently, according to the results from the CASP14 assessment, the DeepMind's latest system AlphaFold, which is much better than its previous version \cite{AlphaFoldP,AlphaFold}, can accurately predict a protein's native structure. However, machine learning methods are data-related. Either rare specific proteins that cannot be fully covered by historical data, or multi-molecule effects that lack sufficient data support, are in its blind spot. Hence, how to improve the traditional forcefield-MD-optimization algorithms to be comparable to AlphaFold is still an important issue. Even more, such an expected approach may also enable relevant researchers to better understand the mechanism of protein folding, meanwhile, this mechanism may also inspire some interesting ideas on global optimization problems, since the folding process usually occurs on timescales of milliseconds to seconds.

\subsection{Main idea and contributions}

Under Assumption \ref{APP:ass:A}, the objective $f$ has a unique global minima $x_*$ and possibly many multiple local minima. The goal here is to find this global minima $x_*$ without being trapped in saddle points, local minima, or even discontinuities. In nonconvex settings, we consider the following proximal point iteration
\begin{equation}\label{APP:eq:PPI}
  x_{k+1}=\arg\min_{x\in\mathbb{R}^d}
  \left(f(x)+\frac{\lambda}{2}\|x-x_k\|_2^2\right)
  ~~\textrm{for a fixed scalar}~~\lambda>0.
\end{equation}
First, we prove that, under the relevant conditions (see Theorem \ref{APP:thm:Ax}), the nonconvex proximal point $x_{k+1}$ can be asymptotically represented as
\begin{equation}\label{APP:eq:RS}
  x_{k+1}=\lim_{\alpha\to\infty}\frac{\int_{\mathbb{R}^d}x\exp
  \big[-\alpha\big(f(x)+\frac{\lambda}{2}\|x-x_k\|_2^2\big)\big]\ud x}
  {\int_{\mathbb{R}^d}\exp\big[-\alpha\big(f(x)+
  \frac{\lambda}{2}\|x-x_k\|_2^2\big)\big]\ud x}.
\end{equation}
This asymptotic representation is the basis for us to deal with nonconvex problems. And this is also the reason that our method is referred to as the \emph{asymptotic proximal point} method. Especially, we consider its asymptotic behavior when $\alpha$ is chosen to be $\alpha_k=\rho^{-k}$ and the two integrals in the ratio are replaced with Monte Carlo estimates.

More clearly, with an initial point $x_1$, three fixed parameters $\lambda>0$, $0<\rho<1$ and $n\in\mathbb{N}$, the original APP methods are characterized by the iteration
\begin{equation}\label{APP:eq:RAI}
  x_{k+1}=\frac{\sum_{i=1}^n\theta_i
  \exp\big(-\rho^{-k}f(\theta_i)\big)}
  {\sum_{i=1}^n\exp\big(-\rho^{-k}f(\theta_i)\big)},~~~
  \theta_i\sim\mathcal{N}(x_k,\rho^{k}\lambda^{-1}I_d),
\end{equation}
and the stable APP methods are characterized by the iteration
\begin{equation}\label{APP:eq:SRAI}
  x_{k+1}=\frac{\sum_{i=1}^n\theta_i \exp[-m_k^{-1}(f(\theta_i)-f_*)]}{\sum_{i=1}^n \exp[-m_k^{-1}(f(\theta_i)-f_*)]},~~~
  \theta_i\sim\mathcal{N}(x_k,\rho^{k}\lambda^{-1}I_d),
\end{equation}
where $I_d\in\mathbb{R}^{d\times d}$ is an identity matrix and $m_k^2=\mathbb{E}[(f(\theta_i)-f_*)^2]$; in practice, $f_*$ and $m_k$ can be replaced with corresponding estimates. More importantly, our main results show that, when the corresponding parameters are properly selected, both the original methods (Theorems \ref{APP:thm:main2} and \ref{APP:thm:main3}) and stable methods (together with Theorem \ref{APP:thm:mk}) enjoy global linear convergence, i.e., $\|x_k-x_*\|_2^2=\mathcal{O}(\rho^k)$, for finding the global minimizer $x_*$ under Assumption \ref{APP:ass:A}.

\subsection{Related Work}

We discuss the relationships between the new method and two closely related ideas, which are the \emph{proximal point methods} and the \emph{asymptotic solution of minimization problems}. Then, we comment some popular methods for finding a first-order critical point or second-order stationary point in a nonconvex setting, including \emph{derivative-based descents}, \emph{perturbed gradient descents} and \emph{derivative-free descents}.

\textbf{Proximal point methods.} The proximal point method (e.g., \cite{PoliakB1987M_Optimization}), which could be traced back to Martinet \cite{MartinetB1970_Proximal} in the context of convex minimization and Rockafellar \cite{RockafellarR1976A_Proximal} in the general setting of maximal monotone operators, is a conceptually simple approach for minimizing a function $f$ on $\mathbb{R}^d$. Given an iterate $x_k$, the method defines $x_{k+1}$ to be any minimizer of the proximal point iteration
\begin{equation}\label{APP:eq:PI}
  \arg\min_{x\in\mathbb{R}^d}
  \Big(f(x)+\frac{\lambda}{2}\|x-x_k\|_2^2\Big)
  ~~\textrm{for an appropriate}~\lambda>0,
\end{equation}
which could be seen as minimizing $\Psi(x_k):=\min_{x}(f(x)+ \frac{\lambda}{2}\|x-x_k\|_2^2)$ by applying the gradient descent with the stepsize $\frac{1}{\lambda}$ \cite{PoliakB1987M_Optimization}. Most of the time, the term proximal refers to the presence of the regularization term with a large $\lambda$, which encourages the new iterate to be close to $x_k$ \cite{BottouL2018R_SGD,ParpasP2017A_ProximalGradient}. However, a feature of our APP methods is that the regularization parameter $\lambda$ should be small enough to ensure a large search range.

The proximal point method is not only the basis of many techniques for convex optimization, but also has a relatively long history of solving nonconvex problems. Obviously, proximal regularization performed with an appropriate parameter ensures convexity of the auxiliary problems for some nonconvex cases \cite{KaplanA1998_PPNO}. More generally, as pointed out by Hare and Sagastiz\'{a}bal \cite{Hare2009_PPNO}, the inexact proximal point methods for convex optimization could be redesigned to deal with nonconvex functions and a crucial step is to design efficient methods of approximating nonconvex proximal points. From this point, corresponding methods have been proposed to approximate the proximal points of nonconvex functions \cite{HareW2010_PBNO,Hare2009_PPNO}, as well as nonsmooth convex functions \cite{Hare2018_PPSubgradient}. Similarly, nonconvex proximal points play an important role in our APP methods. We constructed a sequence of approximate nonconvex proximal points that converges linearly to the global minima of the function satisfied Assumption \ref{APP:ass:A}. However, we are based on an asymptotic representation, not an iterative approach. 

\textbf{Asymptotic solution of minimization problems.} In 1967, Pincus proved that \cite{PincusM1968A_AsymptoticSolution}, if $\Omega\subset \mathbb{R}^d$ is a bounded domain, $f$ is a continuous function on $\Omega$ and has a unique minimizer $s_*$ over $\Omega$, then the minimizer can be represented as
\begin{equation}\label{APP:eq:Pincus}
  s_*=\lim_{\alpha\to\infty}\frac{\int_\Omega x\exp
  \big(-\alpha f(x)\big)\ud x}
  {\int_\Omega \exp\big(-\alpha f(x)\big)\ud x}.
\end{equation}
And later, Pincus \cite{PincusM1970A_AsymptoticSolution} further suggested a Monte Carlo estimate to approximate the minimizer $s_*$. This idea did not receive enough attention because it is not sufficiently efficient \cite{TornA1989M_GlobalOptimization}. The major reason is that, for building such an estimate, one has to keep sampling uniformly on the entire domain $\Omega$, however, when $\alpha$ goes large, the main contributors in these samples are only those sufficiently close to the minimizer. But even so, its extension \cite{ZidaniH2016_GOformula} still obtained valuable applications for global optimization on infinite dimensional Hilbert spaces with finite measures.

In our work, we extend this asymptotic formula \eqref{APP:eq:Pincus} from a different perspective. We consider a similar approach to asymptotically represent the proximal points of nonconvex functions. On the one hand, this nonconvex representation prevents the APP methods from being trapped in saddle points, local minima, or discontinuities. On the other hand, the proximal regularization leads to a normal sampling distribution so that the corresponding samples will gather in the vicinity of the global minimizer as $\alpha$ increases. In addition, this idea is also similar to the viewpoint on stochastic gradient method in \cite{PasupathyR2018_OptimMC}: the increment of each iteration is regarded as an estimate related to Monte Carlo method. Although the estimates we use are quite different, we all achieve a linear convergence by reducing the variance. However, Pasupathy et al \cite{PasupathyR2018_OptimMC} uses a gradually increasing sequence of sample sizes, while we could reduce the variance with a fixed sample size, due to the variance contraction of our proximal point estimates (see Lemma \ref{APP:lem:EV} for details).

\textbf{Derivative-based descent methods.} For convenience we call an $\epsilon$-approximate first-order critical point \emph{$\epsilon$-solution}. It is known that the gradient method could find an $\epsilon$-solution in $\mathcal{O}(\epsilon^{-2})$ iterations for every function $f$ with Lipschitz continuous gradients \cite{NesterovY2018M_ConvexOptimization}. Further, if $f$ additionally has Lipschitz continuous Hessian, then the accelerated gradient method \cite{CarmonY2018A_NonConvex7o4} can achieve the complexity $\mathcal{O}(\epsilon^{-7/4}\log \frac{1}{\epsilon})$; by using Hessians, the cubic regularization of Newton method \cite{NesterovY2006A_CRNewton3o2,CartisC2010A_RegularizedDescent} can find an $\epsilon$-solution in $\mathcal{O}(\epsilon^{-3/2})$ iterations. More generally, the $p$-order regularization methods \cite{BirginE2017A_pOrderRegularized} could find an $\epsilon$-solution in $\mathcal{O}(\epsilon^{-(p+1)/p})$ iterations for every $f$ with Lipschitz continuous derivatives up to order $p\geqslant1$, and this complexity cannot be further improved \cite{CarmonY2019A_LowerBoundsA}. Furthermore, the first-order methods could not achieve the complexity $\mathcal{O}(\epsilon^{-8/5})$ for arbitrarily smooth functions \cite{CarmonY2019A_LowerBoundsB}. These excellent results show that, without additional assumptions, even finding a first-order critical point in nonconvex settings is relatively difficult. And this is one of the reasons that we introduced Assumption \ref{APP:ass:A}.

\textbf{Perturbed gradient descent methods.} In nonconvex settings, convergence to first-order critical points is not yet satisfactory. In 1988, Pemantle \cite{PemantleR1990A_SaddlePoints} realized that by adding zero-mean noise perturbations, a gradient descent method can circumvent strict saddle points with probability one. More recently, it is shown that for all twice differentiable strict saddle functions, the perturbed gradient method converges to an $\epsilon$-second-order stationary point with high probability in $\mathcal{O}(\epsilon^{-2})$ iterations \cite{JinC2017A_EscapeSaddlePoints}. Even without adding noise perturbations, gradient descent with random initialization \cite{LeeJ2016A_SaddlePoints} can also avoid strict saddle points with probability one. The trust region techniques \cite{SunJ2017A_TrustRegion1,SunJ2017A_TrustRegion2} can also avoid saddle points and are not limited to strict saddle points, but generally, the perturbed gradient method is more efficient in practice. And similarly, each APP iteration involves a normal sampling distribution which leads to a global asymptotic property, so it will not be trapped at any saddle point, or even local minima and discontinuities. 

\textbf{Derivative-free descent methods.} The derivative-free descent methods (e.g., \cite{ConnA2009M_DerivativeFree,PoliakB1987M_Optimization}), which are also known as zero-order methods \cite{DuchiJ2015A_ZeroOrderCO,ShamirO2017A_ZeroOrderConvex} in the literature or bandit optimization in the machine learning literature \cite{HazanE2014A_BanditOptimization,ShamirO2017A_ZeroOrderConvex}, were among the first schemes suggested in the early days of the development of optimization theory \cite{MatyasJ1965A_RandomOptimization}. One of the most typical derivative-free methods is established by the finite-difference method \cite{NesterovY2017A_GF,PoliakB1987M_Optimization}, and its descent direction can be seen as an asymptotically unbiased estimate of the smoothed gradient  \cite{NemirovskiA1983M_optimization,NesterovY2017A_GF}. Hence, the finite-difference derivative-free descent (FD-DFD) method can be regarded as a smoothed extension of the gradient descent method. In nonconvex settings, the FD-DFD method can also find an $\epsilon$-solution in $\mathcal{O}(\epsilon^{-2})$ iterations for every function with Lipschitz continuous gradients \cite{NesterovY2017A_GF}. The FD-DFD method can obviously be used to solve nonsmooth problems, and it seems intuitive that a sufficiently large smoothing parameter may help the FD-DFD method to stride saddle points, discontinuities or local minima, but further research is needed.

\subsection{Paper Organization}

The remainder of the paper is organized as follows. In the next section, we establish an asymptotic representation formula of nonconvex proximal points. In Sect. \ref{APP:s3}, we propose original and stable APP methods, and provide insights into the behavior of these methods by establishing their convergence properties and complexity bounds. In Sect. \ref{APP:s4}, we demonstrate the benefits of the stable APP method by several numerical experiments and comparisons in various dimensions from $2$ to $500$. And finally, we draw some conclusions in Sect. \ref{APP:s5}.

\section{Asymptotic representation of nonconvex proximal points}
\label{APP:s2}

In this section, we focus on establishing an asymptotic representation formula for the following nonconvex proximal points
\begin{equation}\label{APP:eq:RMP}
  \arg\min_{x\in\mathbb{R}^d}\left(f(x)+\frac{\lambda}{2}\|x-p\|_2^2\right)
  ~~\textrm{for fixed}~p\in\mathbb{R}^d~\textrm{and}~\lambda>0.
\end{equation}
As mentioned above, Pincus \cite{PincusM1968A_AsymptoticSolution} established an asymptotic formula for the solution of $\min_{x\in\Omega}f(x)$, where $\Omega\subset\mathbb{R}^d$ is a bounded domain. Here we extend his formula to the proximal point methods which includes an additional regularization term.

Suppose that \eqref{APP:eq:RMP} has a unique global minimizer $s_*(p,\lambda)=s_*$ over $\mathbb{R}^d$. If $f$ is further bounded below by a scalar $f_{\inf}$, then both
\begin{equation*}
  \exp\Big[-\alpha\Big(f(x)+\frac{\lambda}{2}\|x-p\|_2^2\Big)\Big]
\end{equation*}
and
\begin{equation*}
  \|x-s_*\|_2^2\exp\Big[-\alpha\Big(f(x)+\frac{\lambda}{2}\|x-p\|_2^2\Big)\Big]
\end{equation*}
are integrable on $\mathbb{R}^d$ for any $\alpha>0$. And for the solution of the regularized problem, we have the following asymptotic formula which can be represented as the limit of the ratio of the two integrals:
\begin{theorem}[nonconvex proximal points]\label{APP:thm:Ax}
Suppose that the proximal point \eqref{APP:eq:RMP} is unique and denoted by $s_*$. If $f$ is bounded from below and continuous at $s_*$, then the proximal point can be represented as
\begin{equation}\label{APP:eq:Ax}
  s_*=\lim_{\alpha\to\infty}\!\frac{\int_{\mathbb{R}^d}x\exp\big[
  -\alpha\big(f(x)+\frac{\lambda}{2}\|x\!-\!p\|_2^2\big)\big]\ud x}
  {\int_{\mathbb{R}^d}\exp\big[-\alpha\big(f(x)+
  \frac{\lambda}{2}\|x\!-\!p\|_2^2\big)\big]\ud x}
  =\lim_{\alpha\to\infty}\!
  \frac{\mathbb{E}\big[\theta\exp\big(-\alpha f(\theta)\big)\big]}
  {\mathbb{E}\big[\exp\big(-\alpha f(\theta)\big)\big]},
\end{equation}
where $\theta\sim\mathcal{N}(p,\alpha^{-1}\lambda^{-1}I_d)$ and $I_d\in\mathbb{R}^{d\times d}$ is an identity matrix.
\end{theorem}
\begin{remark}
The representation establised by Pincus \cite{PincusM1968A_AsymptoticSolution} was restricted to a bounded domain $\Omega$ in order to ensure that these two integrals are integrable. Since including a regularization term, two integrals in \eqref{APP:eq:Ax} are integrable on $\mathbb{R}^d$ for any $\alpha>0$ if $f$ is bounded from below.
\end{remark}
\begin{proof}
For convenience we define $\tau(x)=\exp\big[-\big(f(x) +\frac{\lambda}{2}\|x-p\|_2^2\big)\big]$ and
\begin{equation*}
  m^{(\alpha)}(x)=\frac{\tau^\alpha(x)}{\int_{\mathbb{R}^d} \tau^\alpha(x)\ud x}.
\end{equation*}
Clearly, $m^{(\alpha)}(x)>0$ for all $x\in\mathbb{R}^d$ and $\int_{\mathbb{R}^d}m^{(\alpha)}(x)\ud x=1$, then since $\|\cdot\|_2^2$ is convex, Jensen's inequality gives
\begin{equation}\label{APP:eq:AxT0}
  \left\|\int_{\mathbb{R}^d}x~m^{(\alpha)}(x)\ud x-s_*\right\|_2^2
  \!=\left\|\int_{\mathbb{R}^d}(x-s_*)m^{(\alpha)}(x)\ud x\right\|_2^2
  \!\leqslant\!\int_{\mathbb{R}^d}\!\|x-s_*\|_2^2m^{(\alpha)}(x)\ud x.
\end{equation}
We decompose $\mathbb{R}^d$ into two domains to establish an upper bound for the integral on the right-hand side of the last inequality. For all $\delta>0$ we define the open domain
\begin{equation*}
  \Omega_\delta=\{x\in\mathbb{R}^d:\tau(x)>\tau(s_*)-\delta\}
  ~~\textrm{with its complement}~~
  \Omega'_\delta=\mathbb{R}^d-\Omega_\delta.
\end{equation*}
Since $f$ is continuous at $s_*$, we observe that, for small $\epsilon>0$, there exists ${\delta(\epsilon)}>0$ such that $\mu(\Omega_{\delta(\epsilon)})>0$ and $\|x-s_*\|_2^2<\frac{\epsilon}{2}$ for all $x\in\Omega_{\delta(\epsilon)}$, where $\mu(S)$ is the $d$-dimensional Lebesgue measure of a subset $S\subset\mathbb{R}^d$; and we have
\begin{equation}\label{APP:eq:AxT1}
  \int_{\mathbb{R}^d}\|x-s_*\|_2^2m^{(\alpha)}\!(x)\ud x
  =\int_{\Omega_{\delta(\epsilon)}}\!\!\|x-s_*\|_2^2
  m^{(\alpha)}\!(x)\ud x+\int_{\Omega'_{\delta(\epsilon)}}
  \!\!\|x-s_*\|_2^2m^{(\alpha)}\!(x)\ud x.
\end{equation}
For the first integral on the right-hand side of \eqref{APP:eq:AxT1}, we clearly have
\begin{equation}\label{APP:eq:AxT2}
  \int_{\Omega_{\delta(\epsilon)}}\|x-s_*\|_2^2m^{(\alpha)}(x)\ud x
  <\frac{\epsilon}{2}\int_{\Omega_{\delta(\epsilon)}}m^{(\alpha)}(x)
  \ud x<\frac{\epsilon}{2}\int_{\mathbb{R}^d}m^{(\alpha)}(x)\ud x
  =\frac{\epsilon}{2}.
\end{equation}
For the second integral on the right-hand side of \eqref{APP:eq:AxT1}, we obtain
\begin{align*}
  \int_{\Omega'_{\delta(\epsilon)}}\!\|x-s_*\|_2^2m^{(\alpha)}(x)\ud x
  =\frac{\int_{\Omega'_{\delta(\epsilon)}}\!
  \|x\!-\!s_*\|_2^2\tau^\alpha(x)\ud x}
  {\int_{\mathbb{R}^d}\tau^\alpha(x)\ud x}
  <\frac{\int_{\Omega'_{\delta(\epsilon)}}\!
  \|x\!-\!s_*\|_2^2\tau^\alpha(x)\ud x}
  {\int_{\Omega_{\delta(\epsilon)}}\tau^\alpha(x)\ud x}.
\end{align*}
When $\alpha\lambda/2>1$, $\tau^\alpha(x)\exp(\|x-p\|_2^2)$ and $\|x-s_*\|_2^2 \exp(-\|x-p\|_2^2)$ are integrable on $\mathbb{R}^d$, hence, by the mean value theorem for integrals, there is $\xi\in\Omega'_{\delta(\epsilon)}$ such that
\begin{align*}
  \int_{\Omega'_{\delta(\epsilon)}}\!\!\|x-s_*\|_2^2\tau^\alpha(x)\ud x
  =&\tau^\alpha(\xi)\exp(\|\xi-p\|_2^2)
  \int_{\Omega'_{\delta(\epsilon)}}\!\!\|x-s_*\|_2^2
  \exp(-\|x-p\|_2^2)\ud x \\
  <&\tau^\alpha(\xi)\exp(\|\xi-p\|_2^2)
  \int_{\mathbb{R}^d}\|x-s_*\|_2^2\exp(-\|x-p\|_2^2) \ud x;
\end{align*}
and similarly, there is $\zeta\in\Omega_{\delta(\epsilon)}$ such that
\begin{equation*}
  \int_{\Omega_{\delta(\epsilon)}}\tau^\alpha(x)\ud x
  =\tau^\alpha(\zeta)\mu(\Omega_{\delta(\epsilon)}).
\end{equation*}
Thus, we obtain
\begin{equation*}
  \int_{\Omega'_{\delta(\epsilon)}}\!\|x-s_*\|_2^2m^{(\alpha)}(x)\ud x
  <\left(\frac{\tau(\xi)}{\tau(\zeta)}\right)^\alpha
  \frac{\exp(\|\xi-p\|_2^2)I_{s_*,p}}
  {\mu(\Omega_{\delta(\epsilon)})},
\end{equation*}
where $\tau(\xi)<\tau(\zeta)$ and $I_{s_*,p}=\int_{\mathbb{R}^d} \|x-s_*\|_2^2\exp(-\|x-p\|_2^2)\ud x<\infty$. Therefore, there exists a fixed $\alpha_\epsilon>0$ such that for every $\alpha>\alpha_\epsilon$, it holds that
\begin{equation}\label{APP:eq:AxT3}
  \int_{\Omega'_{\delta(\epsilon)}}\|x-s_*\|_2^2m^{(\alpha)}(x)\ud x
  <\frac{\epsilon}{2}.
\end{equation}
Finally, from \eqref{APP:eq:AxT0} - \eqref{APP:eq:AxT3}, we observe that for small $\epsilon>0$, there is a fixed $\alpha_\epsilon>0$ such that for all $\alpha>\alpha_\epsilon$, it holds that
\begin{equation*}
  \left\|\int_{\mathbb{R}^d}x\cdot m^{(\alpha)}(x)
  \ud x-s_*\right\|_2^2<\epsilon,
\end{equation*}
and the proof is complete.\qed
\end{proof}

\section{APP methods}
\label{APP:s3}

In this section, we first propose original APP methods and provide insights into the behavior of the methods by establishing its convergence property and complexity bound; then we also propose stable APP methods, and further show that the analyses of the original APP methods can be directly applied to the stable APP methods.

\subsection{Original idea}

Inspired by the explicit asymptotic formula built in the previous section, our APP methods are procedures in which each iterate is chosen as a weighted average of normally distributed samples with mean equal to the latest iterate. Specifically, with an initial point $x_1$, a fixed regularization parameter $\lambda>0$, a fixed contraction factor $0<\rho<1$ and $n\in\mathbb{N}$, the original APP methods are characterized by the iteration
\begin{equation}\label{APP:eq:RAIxk}
  x_{k+1}=\frac{\sum_{i=1}^n\theta_i
  \exp\big(-\rho^{-k}f(\theta_i)\big)}
  {\sum_{i=1}^n\exp\big(-\rho^{-k}f(\theta_i)\big)},
  ~~\textrm{where}~~\theta_i\sim\mathcal{N}
  (x_k,\rho^{k}\lambda^{-1}I_d).
\end{equation}
Here, $I_d\in\mathbb{R}^{d\times d}$ is an identity matrix so that $\mathcal{N}(x_k,\rho^k\lambda^{-1}I_d)$ is a spherical normal distribution. The iteration \eqref{APP:eq:RAIxk} is derivative-free and can be applied to nonsmooth or even discontinuous problems. Specifically, we can define the original APP method as Algorithm \ref{APP:alg:APP}. The regularization parameter $\lambda$ determines the initial exploration radius and $n$ determines the number of function evaluations per-iteration.

\begin{algorithm}
\caption{Original APP Method}
\label{APP:alg:APP}
\begin{algorithmic}[1]
\STATE{Choose an initial iterate $x_1$ and preset parameters $\lambda>0$, $\rho\in(0,1)$, $n\in\mathbb{N}$.}
\FOR{$k=1,2,\cdots$}
\STATE{Generate $n$ realizations $\{\theta_i\}_i^n$ of the random vector from $\mathcal{N}(x_k,\rho^{k}\lambda^{-1}I_d)$.}
\STATE{Compute function value sequence $\{f(\theta_i)\}_i^n$.}
\STATE{Set the new iterate as $x_{k+1}=\frac{\sum_{i=1}^n\theta_i \exp(-\rho^{-k}f(\theta_i))}{\sum_{i=1}^n \exp(-\rho^{-k}f(\theta_i))}$.}
\ENDFOR
\end{algorithmic}
\end{algorithm}

To establish this iteration, we replaced the two integrals of the ratio in \eqref{APP:eq:Ax} with Monte Carlo estimates, that is,
\begin{equation*}
  \frac{1}{n}\!\sum_{i=1}^{n}\theta_i~
  \exp\big(\!-\rho^{-k}\!f(\theta_i)\big)\approx
  \mathbb{E}\big[\theta\exp\big(-\rho^{-k}f(\theta)\big)\big]
\end{equation*}
and
\begin{equation*}
  \frac{1}{n}\!\sum_{i=1}^{n}
  \exp\big(\!-\rho^{-k}\!f(\theta_i)\big)\approx
  \mathbb{E}\big[\exp\big(-\rho^{-k}f(\theta)\big)\big],
\end{equation*}
where $\theta\sim\mathcal{N}(x_k,\rho^k\lambda^{-1}I_d)$. Hence, $x_{k+1}$ is a sample-based estimate of the exact asymptotic proximal point iteration
\begin{equation*}
  e_{k+1}:=\frac{\mathbb{E}\big[\theta\exp\big(-\rho^{-k}f(\theta)\big)\big]}
  {\mathbb{E}\big[\exp\big(-\rho^{-k}f(\theta)\big)\big]}
  =\frac{\int_{\mathbb{R}^d}x~\exp\big[-\rho^{-k}
  \big(f(x)+\frac{\lambda}{2}\|x-x_k\|_2^2\big)\big]\ud x}
  {\int_{\mathbb{R}^d}\exp\big[-\rho^{-k}\big(f(x)+
  \frac{\lambda}{2}\|x-x_k\|_2^2\big)\big]\ud x}.
\end{equation*}
It is worth pointing out that, although containing the idea of Monte Carlo methods, our method is not directly dependent on its slow convergence rate. Thus, our iteration method described as Algorithm \ref{APP:alg:APP} is much different from the simple Monte Carlo estimator suggested by Pincus \cite{PincusM1970A_AsymptoticSolution}.

As we mentioned above, the APP method enjoys the global linear convergence. Usually, a typical linear convergence can be described as
\begin{equation*}
  \|x_{k+1}-x_*\|_2^2\leqslant\rho\|x_k-x_*\|_2^2
  ~~~\textrm{for a certain}~~~0<\rho<1,
\end{equation*}
which represents a contraction relationship between $\|x_{k+1}-x_*\|_2^2$ and $\|x_k-x_*\|_2^2$ for all $k\in\mathbb{N}$. In the following, we will introduce a critical medium $U_k$ to establish such a contraction relationship. For convenience we first define
\begin{equation}\label{APP:eq:phipsi}
  \phi_k(\theta)=\|\theta-x_*\|_2\exp(-\rho^{-k}f(\theta)) ~~~\textrm{and}~~~\psi_k(\theta)=\exp(-\rho^{-k}f(\theta))
\end{equation}
with $\theta\sim\mathcal{N}(x_k,\rho^k\lambda^{-1}I_d)$, then the medium $U_k$ can be define as
\begin{equation}\label{APP:eq:Ik}
  U_k:=\frac{\mathbb{E}[\phi_k]}{\mathbb{E}[\psi_k]}
  =\frac{\int_\Theta\phi_k(\theta)\ud P_\Theta(\theta)}
  {\int_\Theta\psi_k(\theta)\ud P_\Theta(\theta)};
\end{equation}
and equivalently, \eqref{APP:eq:Ik} can also be rewritten as
\begin{equation}\label{APP:eq:IkR}
  U_k=\frac{\int_{\mathbb{R}^d}\|x-x_*\|_2
  \exp\big[-\rho^{-k}\big(f(x)+\frac{\lambda}{2}
  \|x-x_k\|_2^2\big)\big]\ud x}
  {\int_{\mathbb{R}^d}\exp\big[-\rho^{-k}\big(f(x)+
  \frac{\lambda}{2}\|x-x_k\|_2^2\big)\big]\ud x}.
\end{equation}
Since $\|\cdot\|_2$ is convex, Jensen's inequality gives
\begin{equation}\label{APP:eq:IkRB}
  \|e_{k+1}-x_*\|_2=\left\|\frac{\int_{\mathbb{R}^d}x~\exp\big[-\rho^{-k}
  \big(f(x)+\frac{\lambda}{2}\|x-x_k\|_2^2\big)\big]\ud x}
  {\int_{\mathbb{R}^d}\exp\big[-\rho^{-k}\big(f(x)+
  \frac{\lambda}{2}\|x-x_k\|_2^2\big)\big]\ud x}-x_*\right\|_2
  \leqslant U_k.
\end{equation}
Hence, we will first show that the difference between $\|x_{k+1}-x_*\|_2$ and $\|e_{k+1}-x_*\|_2$ can be controlled by the sample size $n$, and then establish the relationship between $\|e_{k+1}-x_*\|_2$ and $\|x_k-x_*\|_2$.

\subsection{Preliminary}

Before we start, we pause to introduce two $d$-dimensional integrals that occur many times in subsequent analysis. We will see that these results are fundamental to the analyses of APP algorithms. For a proof, see Appendix A.

\begin{lemma}\label{APP:lem:2int}
For any $\alpha,\beta,\gamma\in\mathbb{R}$ and $u,v\in\mathbb{R}^d$, if $\alpha(\beta+\gamma)>0$ and
\begin{equation*}
  \varphi(x)=\exp\left[-\alpha
  \left(\frac{\beta}{2}\|x-u\|_2^2+
  \frac{\gamma}{2}\|x-v\|_2^2\right)\right],
\end{equation*}
then the integrals
\begin{equation*}
  \int_{\mathbb{R}^d}\varphi(x)\ud x=\exp\left[
  -\frac{\alpha\beta\gamma\|u-v\|_2^2}{2(\beta+\gamma)}\right]
  \left[\frac{2\pi}{\alpha(\beta+\gamma)}\right]^{\frac{d}{2}}
\end{equation*}
and
\begin{equation*}
  \int_{\mathbb{R}^d}\!\|x-u\|_2^2\varphi(x)\ud x
  =\exp\left[-\frac{\alpha\beta\gamma\|u\!-\!v\|_2^2}
  {2(\beta+\gamma)}\right]\!\left[\frac{2\pi}
  {\alpha(\beta\!+\!\gamma)}\right]^{\frac{d}{2}}
  \!\!\left[\frac{d}{\alpha(\beta\!+\!\gamma)}
  +\frac{\gamma^2\|u\!-\!v\|_2^2}{(\beta+\gamma)^2}\right].
\end{equation*}
\end{lemma}

\subsection{Three fundamental lemmas}

The following lemma first establishes an upper bound for $\|x_{k+1}-x_*\|_2^2$.
\begin{lemma}[Relationship between $\|x_{k+1}-x_*\|_2^2$ and $U_k^2$]
\label{APP:lem:N2I}
Suppose that the APP method (Algorithm \ref{APP:alg:APP}) is run with a natural number $n$ such that, for all $k\in\mathbb{N}$ and $C>0$, the number of function evaluations per-iteration
\begin{equation}\label{APP:eq:nk1}
  n\geqslant n_k:=\frac{4C^2\mathbb{V}[\psi_k]}{(\mathbb{E}[\psi_k])^2}.
\end{equation}
Then with probability at least $1-\frac{1}{C^2}$, the iterates of APP satisfy for all $k\in\mathbb{N}$:
\begin{equation*}
  \|x_{k+1}-x_*\|_2^2\leqslant8U_k^2+8\frac{C^2\mathbb{V}[\phi_k]}
  {(\mathbb{E}[\psi_k])^2}\frac{1}{n},
\end{equation*}
where $U_k$ is defined by \eqref{APP:eq:IkR}, $\phi_k$ and $\psi_k$ are defined by \eqref{APP:eq:phipsi}.
\end{lemma}
\begin{proof}
It follows from the iteration \eqref{APP:eq:RAIxk} that
\begin{align*}
  x_{k+1}-x_*=\frac{\sum_{i=1}^{n_k}(\theta_i-x_*)
  \exp\big(-\rho^{-k}f(\theta_i)\big)}
  {\sum_{i=1}^{n_k}\exp\big(-\rho^{-k}f(\theta_i)\big)},
\end{align*}
since $\|\cdot\|_2$ is convex, Jensen's inequality gives
\begin{align*}
  \|x_{k+1}-x_*\|_2=&\left\|\frac{\sum_{i=1}^{n_k}(\theta_i-x_*)
  \exp\big(-\rho^{-k}f(\theta_i)\big)}{\sum_{i=1}^{n_k}
  \exp\big(-\rho^{-k}f(\theta_i)\big)}\right\|_2 \\
  \leqslant&\frac{\sum_{i=1}^{n_k}\|\theta_i-x_*\|_2
  \exp\big(-\rho^{-k}f(\theta_i)\big)}{\sum_{i=1}^{n_k}
  \exp\big(-\rho^{-k}f(\theta_i)\big)},
\end{align*}
which can be rewritten as
\begin{equation}\label{APP:eq:mpp1}
  \|x_{k+1}-x_*\|_2\leqslant\frac{\bar{\phi}_k}{\bar{\psi}_k},
  ~~\textrm{where}~~\bar{\phi}_k=\frac{1}{n_k}
  \sum_{i=1}^{n_k}\phi_k(x_i),~~
  \bar{\psi}_k=\frac{1}{n_k}\sum_{i=1}^{n_k}\psi_k(x_i).
\end{equation}
Furthermore, notice that
\begin{equation*}
  \mathbb{E}[\bar{\phi}_k]=\mathbb{E}[\phi_k],~~~
  \mathbb{V}[\bar{\phi}_k]=\mathbb{V}[\phi_k]/n_k,
\end{equation*}
and
\begin{equation*}
  \mathbb{E}[\bar{\psi}_k]=\mathbb{E}[\psi_k],~~~
  \mathbb{V}[\bar{\psi}_k]=\mathbb{V}[\psi_k]/n_k,
\end{equation*}
it follows from Chebyshev's inequality that, for all $C>0$,
\begin{equation*}
  \mathbb{P}\left(|\bar{\phi}_k-\mathbb{E}[\phi_k]|\geqslant
  C\sqrt{\mathbb{V}[\phi_k]/n_k}\right)\leqslant\frac{1}{C^2}
\end{equation*}
and
\begin{equation*}
  \mathbb{P}\left(|\bar{\psi}_k-\mathbb{E}[\psi_k]|\geqslant
  C\sqrt{\mathbb{V}[\psi_k]/n_k}\right)\leqslant\frac{1}{C^2},
\end{equation*}
that is, with probability at least $1-\frac{1}{C^2}$, it hold that
\begin{equation*}
  |\bar{\phi}_k-\mathbb{E}[\phi_k]|\leqslant
  C\sqrt{\mathbb{V}[\phi_k]/n_k},~~~
  |\bar{\psi}_k-\mathbb{E}[\psi_k]|\leqslant
  C\sqrt{\mathbb{V}[\psi_k]/n_k},
\end{equation*}
and further,
\begin{align*}
  \frac{\bar{\phi}_k}{\bar{\psi}_k}\leqslant&
  \frac{\mathbb{E}[\phi_k]+C\sqrt{\mathbb{V}[\phi_k]/n_k}}
  {\mathbb{E}[\psi_k]-C\sqrt{\mathbb{V}[\psi_k]/n_k}}
  =\frac{\mathbb{E}[\phi_k]+C\sqrt{\mathbb{V}[\phi_k]/n_k}}
  {\mathbb{E}[\psi_k]}\frac{\mathbb{E}[\psi_k]}{\mathbb{E}[\psi_k]
  -C\sqrt{\mathbb{V}[\psi_k]/n_k}} \\
  =&\bigg(U_k+\frac{C\sqrt{\mathbb{V}[\phi_k]}}
  {\mathbb{E}[\psi_k]\sqrt{n_k}}\bigg)\frac{\sqrt{n_k}}
  {\sqrt{n_k}-C\sqrt{\mathbb{V}[\psi_k]}/\mathbb{E}[\psi_k]},
\end{align*}
since $\frac{\sqrt{t}}{\sqrt{t}-s}$ is monotonically decreasing with respect to $t$ when $t>s$ for any $s\in\mathbb{R}$, the condition \eqref{APP:eq:nk1} guarantees that
\begin{equation*}
  \frac{\sqrt{n_k}}{\sqrt{n_k}-C\sqrt{\mathbb{V}[\psi_k]}
  /\mathbb{E}[\psi_k]}\leqslant2,
\end{equation*}
thus, we obtain
\begin{equation}\label{APP:eq:mpp2}
  \frac{\bar{\phi}_k}{\bar{\psi}_k}\leqslant2U_k+2
  \frac{C\sqrt{\mathbb{V}[\phi_k]}}{\mathbb{E}[\psi_k]}
  \frac{1}{\sqrt{n_k}}.
\end{equation}
By noting that the Arithmetic Mean Geometric Mean inequality, it follows from \eqref{APP:eq:mpp1} and \eqref{APP:eq:mpp2} that, with probability at least $1-\frac{1}{C^2}$,
\begin{align*}
  \|x_{k+1}-x_*\|_2^2\leqslant\left(
  \frac{\bar{\phi}_k}{\bar{\psi}_k}\right)^2\leqslant
  \left(2U_k+2\frac{C\sqrt{\mathbb{V}[\phi_k]}}
  {\mathbb{E}[\psi_k]}\frac{1}{\sqrt{n_k}}\right)^2\leqslant8U_k^2+8
  \frac{C^2\mathbb{V}[\phi_k]}{(\mathbb{E}[\psi_k])^2}\frac{1}{n_k},
\end{align*}
and the proof is complete.\qed
\end{proof}

The following lemma shows the relationship between $U_k^2$ and $\|x_k-x_*\|_2^2$ under certain conditions. We will see later that these conditions are very easy to meet in the Algorithm \ref{APP:alg:APP}.
\begin{lemma}[Relationship between $U_k^2$ and $\|x_k-x_*\|_2^2$]
\label{APP:lem:I2N}
Under Assumption \ref{APP:ass:A}, suppose that the APP method (Algorithm \ref{APP:alg:APP}) is run with a regularization parameter $\lambda>0$ and a contraction factor $0<\rho<1$ such that
\begin{equation}\label{APP:eq:rholambda}
  \rho_\lambda=\frac{10\lambda^2L^{\frac{d}{2}}}{l^{\frac{d}{2}+2}} \exp\left(\frac{\lambda^2M}{2l}\right)<1
\end{equation}
and $\|x_k-x_*\|_2^2\leqslant\rho^kM$ for a fixed $M>0$ and $k\in\mathbb{N}$. Then the iterates of APP satisfy the following inequality:
\begin{equation*}
  U_k^2\leqslant\rho^k\frac{\rho_\lambda}{10}\frac{ld}{\lambda^2}
  +\frac{\rho_\lambda}{10}\|x_k-x_*\|_2^2,
\end{equation*}
where $U_k$ is defined by \eqref{APP:eq:IkR}.
\end{lemma}
\begin{remark}
Note that for any $0<\epsilon<1$, there is $\lambda_\epsilon>0$ such that for every $\lambda<\lambda_\epsilon$, it holds that $\rho_\lambda<\epsilon$.
\end{remark}
\begin{proof}
According to Jensen¡¯s inequality for convex functions, it holds that
\begin{equation}\label{APP:eq:I2Nt1}
  U_k^2\leqslant
  \frac{\int_{\mathbb{R}^d}\|x\!-\!x_*\|_2^2\exp\big[-\rho^{-k}
  \big(f(x)+\frac{\lambda}{2}\|x-x_k\|_2^2\big)\big]\ud x}
  {\int_{\mathbb{R}^d}\exp\big[-\rho^{-k}\big(f(x)+
  \frac{\lambda}{2}\|x-x_k\|_2^2\big)\big]\ud x},
\end{equation}
further, from Assumption \ref{APP:ass:A}, i.e.,
\begin{equation*}
  f_*+\frac{l}{2}\|x-x_*\|_2^2\leqslant
  f(x)\leqslant f_*+\frac{L}{2}\|x-x_*\|_2^2,
\end{equation*}
one can first observe that, for the fraction on the right-hand side of \eqref{APP:eq:I2Nt1}, we have the upper bound of the numerator
\begin{equation}\label{APP:eq:I2Nt2}
  \exp(-\rho^{-k}f_*)\int_{\mathbb{R}^d}\|x-x_*\|_2^2
  \exp\Big[-\rho^{-k}\Big(\frac{l}{2}\|x-x_*\|_2^2+\frac{\lambda}{2}
  \|x-x_k\|_2^2\Big)\Big]\ud x
\end{equation}
and the lower bound of the denominator
\begin{equation}\label{APP:eq:I2Nt3}
\begin{split}
  \exp(-\rho^{-k}f_*)\int_{\mathbb{R}^d}
  \exp\Big[-\rho^{-k}\Big(\frac{L}{2}\|x-x_*\|_2^2
  +\frac{\lambda}{2}\|x-x_k\|_2^2\Big)\Big]\ud x.
\end{split}
\end{equation}
So it follows from \eqref{APP:eq:I2Nt1} - \eqref{APP:eq:I2Nt3} that
\begin{equation}\label{APP:eq:I2Nt4}
  U_k^2\leqslant\frac{\int_{\mathbb{R}^d}\|x-x_*\|_2^2
  \exp\big[-\rho^{-k}\big(\frac{l}{2}\|x-x_*\|_2^2+\frac{\lambda}{2}
  \|x-x_k\|_2^2\big)\big]\ud x}
  {\int_{\mathbb{R}^d}\exp\big[-\rho^{-k}
  \big(\frac{L}{2}\|x-x_*\|_2^2+\frac{\lambda}{2}
  \|x-x_k\|_2^2\big)\big]\ud x}.
\end{equation}
Further, according to Lemma \ref{APP:lem:2int}, the numerator of the fraction on the right-hand side of \eqref{APP:eq:I2Nt4} equals to
\begin{equation*}
  \exp\left(-\frac{\rho^{-k}l\lambda\|x_k-x_*\|_2^2}
  {2(l+\lambda)}\right)\left(\frac{2\pi\rho^k}
  {l\!+\!\lambda}\right)^{\frac{d}{2}}
  \!\left(\frac{\rho^kd}{l\!+\!\lambda}
  +\frac{\lambda^2\|x_k-x_*\|_2^2}{(l+\lambda)^2}\right)
\end{equation*}
and the corresponding denominator equals to
\begin{equation*}
  \exp\left(-\frac{\rho^{-k}L\lambda\|x_k-x_*\|_2^2}
  {2(L+\lambda)}\right)\left(\frac{2\pi\rho^k}
  {L+\lambda}\right)^{\frac{d}{2}}.
\end{equation*}
Thus, by noting that $\|x_k-x_*\|_2^2\leqslant\rho^kM$ and
\begin{equation*}
  \rho_\lambda=\frac{10\lambda^2L^{\frac{d}{2}}}{l^{\frac{d}{2}+2}} \exp\left(\frac{\lambda^2M}{2l}\right)<1,
\end{equation*}
we obtain
\begin{align*}
  U_k^2\leqslant&\exp\left(\frac{\lambda^2(L-l)
  \rho^{-k}\|x_k-x_*\|_2^2}{2(L+\lambda)(l+\lambda)}\right)\!
  \left(\frac{L+\lambda}{l+\lambda}\right)^{\frac{d}{2}}\!\!
  \left(\frac{\rho^kd}{l+\lambda}
  +\frac{\lambda^2\|x_k-x_*\|_2^2}{(l+\lambda)^2}\right) \\
  \leqslant&\exp\left(\frac{\lambda^2(L-l)M}{2Ll}\right)\!
  \left(\frac{L}{l}\right)^{\frac{d}{2}}\!\!
  \left(\frac{\rho^kd}{l}
  +\frac{\lambda^2\|x_k-x_*\|_2^2}{l^2}\right) \\
  \leqslant&\rho^k\frac{\rho_\lambda}{10}\frac{ld}{\lambda^2}
  +\frac{\rho_\lambda}{10}\|x_k-x_*\|_2^2,
\end{align*}
and the proof is complete.\qed
\end{proof}

Now we turn to two quantities, i.e., $\frac{\mathbb{V}[\psi_k]}
{(\mathbb{E}[\psi_k])^2}$ and $\frac{\mathbb{V}[\phi_k]}
{(\mathbb{E}[\psi_k])^2}$, appearing in Lemma \ref{APP:lem:N2I}. The following lemma gives their upper bounds. And this is the last of the three fundamental lemmas.
\begin{lemma}[Upper bounds for $\frac{\mathbb{V}[\psi_k]}
{(\mathbb{E}[\psi_k])^2}$ and $\frac{\mathbb{V}[\phi_k]}
{(\mathbb{E}[\psi_k])^2}$]\label{APP:lem:EV}
Suppose the conditions of Lemma \ref{APP:lem:I2N} hold. The APP iterates (Algorithm \ref{APP:alg:APP}) satisfy that, for all $k\in\mathbb{N}$,
\begin{equation*}
  \frac{\mathbb{V}[\psi_k]}{(\mathbb{E}[\psi_k])^2}\leqslant
  \frac{\rho_\lambda l^2}{10\lambda^2}\frac{(L+\lambda)^d}
  {2^{\frac{d}{2}}\lambda^{\frac{d}{2}}L^{\frac{d}{2}}}
  \exp\bigg(\frac{\lambda M}{2}\bigg)
\end{equation*}
and
\begin{equation*}
  \frac{\mathbb{V}[\phi_k]}{(\mathbb{E}[\psi_k])^2}
  \leqslant\frac{(L+\lambda)^d}
  {2^{\frac{d}{2}}\lambda^{\frac{d}{2}}L^{\frac{d}{2}}}
  \exp\bigg(\frac{\lambda M}{2}\bigg)
  \left(\rho^k\frac{\rho_\lambda}{20}\frac{ld}{\lambda^2}
  +\frac{\rho_\lambda}{40}\|x_k-x_*\|_2^2\right).
\end{equation*}
\end{lemma}
\begin{proof}
Noting that
\begin{equation*}
  \frac{\mathbb{V}[\psi_k]}{(\mathbb{E}[\psi_k])^2}\leqslant
  \frac{\mathbb{E}[\psi_k^2]}{(\mathbb{E}[\psi_k])^2}~~\textrm{and}
  ~~\frac{\mathbb{V}[\phi_k]}{(\mathbb{E}[\psi_k])^2}\leqslant
  \frac{\mathbb{E}[\phi_k^2]}{(\mathbb{E}[\psi_k])^2},
\end{equation*}
we need to establish a lower bound for $\mathbb{E}[\psi_k]$ and upper bounds for $\mathbb{E}[\psi_k^2]$ and $\mathbb{E}[\phi_k^2]$. And, of course, these upper bound are crude relatively.

We first establish a lower bound for $\mathbb{E}[\psi_k]$. Note that
\begin{align*}
  \mathbb{E}[\psi_k]=&\int_\Theta
  \exp\big(-\rho^{-k}f(\theta)\big)\ud P_\Theta(\theta) \\
  =&\left(\frac{\rho^{-k}\lambda}{2\pi}\right)^{\frac{d}{2}}
  \int_{\mathbb{R}^d}\exp\Big[-\rho^{-k}\Big(f(x)+\frac{\lambda}{2}
  \|x-x_k\|_2^2\Big)\Big]\ud x.
\end{align*}
together with Assumption \ref{APP:ass:A}, then yields
\begin{equation*}
  \mathbb{E}[\psi_k]\geqslant E_k
  \left(\frac{\rho^{-k}\lambda}{2\pi}\right)^{\frac{d}{2}}
  \int_{\mathbb{R}^d}\exp\Big[-\rho^{-k}
  \Big(\frac{L}{2}\|x-x_*\|_2^2+\frac{\lambda}{2}
  \|x-x_k\|_2^2\Big)\Big]\ud x,
\end{equation*}
where $E_k=\exp(-\rho^{-k}f_*)$; further, according to Lemma \ref{APP:lem:2int}, one obtains
\begin{equation}\label{APP:eq:NB1}
  \mathbb{E}[\psi_k]\geqslant E_k
  \left(\frac{\lambda}{L+\lambda}\right)^{\frac{d}{2}}\exp
  \left(-\frac{\rho^{-k}L\lambda\|x_k-x_*\|_2^2}{2(L+\lambda)}\right).
\end{equation}

Now we establish an upper bound for $\mathbb{E}[\psi_k^2]$. Note that
\begin{align*}
  \mathbb{E}[\psi_k^2]=&\int_\Theta
  \exp\big(-2\rho^{-k}f(\theta)\big)\ud P_\Theta(\theta) \\
  =&\left(\frac{\rho^{-k}\lambda}{2\pi}\right)^{\frac{d}{2}}
  \int_{\mathbb{R}^d}\exp\Big[-\rho^{-k}\Big(2f(x)
  +\frac{\lambda}{2}\|x-x_k\|_2^2\Big)\Big]\ud x,
\end{align*}
together with Assumption \ref{APP:ass:A} and Lemma \ref{APP:lem:2int}, then yields
\begin{align}
  \mathbb{E}[\psi_k^2]\leqslant&E_k^2
  \left(\frac{\rho^{-k}\lambda}{2\pi}\right)^{\frac{d}{2}}
  \int_{\mathbb{R}^d}\exp\Big[-\rho^{-k}
  \Big(l\|x-x_*\|_2^2+\frac{\lambda}{2}
  \|x-x_k\|_2^2\Big)\Big]\ud x \notag\\
  =&E_k^2\left(\frac{\lambda}{2l+\lambda}\right)^{\frac{d}{2}}
  \exp\left(-\frac{\rho^{-k}l\lambda\|x_k-x_*\|_2^2}{2l+\lambda}\right).
  \label{APP:eq:NB2}
\end{align}
Therefore, according to \eqref{APP:eq:NB1} and \eqref{APP:eq:NB2}, and together with
\begin{equation}\label{APP:eq:NB3}
  \exp\bigg(\frac{\rho^{-k}L\lambda\|x_k-x_*\|_2^2}{L+\lambda}
  -\frac{\rho^{-k}l\lambda\|x_k-x_*\|_2^2}{2l+\lambda}\bigg)\leqslant
  \exp\bigg(\frac{\lambda M\big(Ll+(L-l)\lambda\big)}{2Ll}\bigg),
\end{equation}
we get the following bound
\begin{align*}
  \frac{\mathbb{V}[\psi_k]}{(\mathbb{E}[\psi_k])^2}\leqslant&
  \left(\frac{(L+\lambda)^2}{\lambda(2l+\lambda)}\right)^{\frac{d}{2}}
  \exp\bigg(\frac{\lambda M\big(Ll+(L-l)\lambda\big)}{2Ll}\bigg) \\
  \leqslant&\frac{\rho_\lambda l^2}{10\lambda^2}\frac{(L+\lambda)^d}
  {2^{\frac{d}{2}}\lambda^{\frac{d}{2}}L^{\frac{d}{2}}}
  \exp\bigg(\frac{\lambda M}{2}\bigg).
\end{align*}

Similarly, we establish an upper bound for $\mathbb{E}[\phi_k^2]$. Note that
\begin{align*}
  \mathbb{E}[\phi_k^2]=&\int_\Theta\|\theta-x_*\|_2^2
  \exp\big(-2\rho^{-k}f(\theta)\big)\ud P_\Theta(\theta) \\
  =&\left(\frac{\rho^{-k}\lambda}{2\pi}\right)^{\frac{d}{2}}
  \int_{\mathbb{R}^d}\|x-x_*\|_2^2\exp\Big[-\rho^{-k}
  \Big(2f(x)+\frac{\lambda}{2}\|x-x_k\|_2^2\Big)\Big]\ud x,
\end{align*}
together with Assumption \ref{APP:ass:A} and Lemma \ref{APP:lem:2int}, this yields
\begin{align*}
  \mathbb{E}[\phi_k^2]\leqslant&E_k^2
  \left(\frac{\rho^{-k}\lambda}{2\pi}\right)^{\frac{d}{2}}
  \int_{\mathbb{R}^d}\|x-x_k\|_2^2\exp\Big[-\rho^{-k}
  \Big(l\|x-x_*\|_2^2+\frac{\lambda}{2}
  \|x-x_k\|_2^2\Big)\Big]\ud x \\
  =&E_k^2\left(\frac{\lambda}{2l+\lambda}\right)^{\frac{d}{2}}\exp
  \left(-\frac{\rho^{-k}l\lambda\|x_k-x_*\|_2^2}{2l+\lambda}\right)
  \left(\frac{\rho^kd}{2l+\lambda}
  +\frac{\lambda^2\|x_k-x_*\|_2^2}{(2l+\lambda)^2}\right),
\end{align*}
Therefore, together with \eqref{APP:eq:NB1} and \eqref{APP:eq:NB3}, we get
\begin{align*}
  \frac{\mathbb{V}[\phi_k]}{(\mathbb{E}[\psi_k])^2}\leqslant&
  \left(\frac{(L+\lambda)^2}{\lambda(2l+\lambda)}\right)^{\frac{d}{2}}
  \exp\bigg(\frac{\lambda M\big(Ll+(L-l)\lambda\big)}{2Ll}\bigg)
  \left(\frac{\rho^kd}{2l+\lambda}
  +\frac{\lambda^2\|x_k-x_*\|_2^2}{(2l+\lambda)^2}\right) \\
  \leqslant&\rho^k\rho_\lambda\frac{l(L+\lambda)^dd}
  {2^{\frac{d}{2}+1}\lambda^{\frac{d}{2}+2}L^{\frac{d}{2}}}
  \exp\bigg(\frac{\lambda M}{2}\bigg)
  +\rho_\lambda\frac{(L+\lambda)^d}
  {2^{\frac{d}{2}+2}\lambda^{\frac{d}{2}}L^{\frac{d}{2}}}
  \exp\bigg(\frac{\lambda M}{2}\bigg)\|x_k-x_*\|_2^2 \\
  \leqslant&\frac{(L+\lambda)^d}
  {2^{\frac{d}{2}}\lambda^{\frac{d}{2}}L^{\frac{d}{2}}}
  \exp\bigg(\frac{\lambda M}{2}\bigg)
  \left(\rho^k\frac{\rho_\lambda}{20}\frac{ld}{\lambda^2}
  +\frac{\rho_\lambda}{40}\|x_k-x_*\|_2^2\right),
\end{align*}
and the proof is complete.\qed
\end{proof}

\subsection{Main results}

First, the expected relationship between $\|x_{k+1}-x_*\|_2^2$ and $\|x_k-x_*\|_2^2$ for a fixed $k\in\mathbb{N}$ is established in the following theorem.
\begin{theorem}[Convergence of one-step iteration]
\label{APP:thm:main1}
Under Assumption \ref{APP:ass:A}, suppose that the APP method (Algorithm \ref{APP:alg:APP}) is run with a regularization parameter $\lambda>0$, a contraction factor $0<\rho<1$ and $n\in\mathbb{N}$ such that $\|x_k-x_*\|_2^2\leqslant\rho^kM$ for a fixed $M>0$ and $k\in\mathbb{N}$,
\begin{equation*}
  \rho_\lambda=\frac{10\lambda^2L^{\frac{d}{2}}}{l^{\frac{d}{2}+2}} \exp\left(\frac{\lambda^2M}{2l}\right)<1,
\end{equation*}
and
\begin{equation*}
  n\geqslant\frac{C^2(L+\lambda)^d}
  {2^{\frac{d}{2}}\lambda^{\frac{d}{2}}L^{\frac{d}{2}}}
  \exp\bigg(\frac{\lambda M}{2}\bigg)
  \max\left\{1,\frac{2l^2}{5\lambda^2}\right\}.
\end{equation*}
Then with probability at least $1-\frac{1}{C^2}$, the iterates of APP satisfy:
\begin{equation}\label{APP:eq:main1}
  \|x_{k+1}-x_*\|_2^2\leqslant\rho^k\rho_\lambda
  \frac{6ld}{5\lambda^2}+\rho_\lambda\|x_k-x_*\|_2^2.
\end{equation}
\end{theorem}
\begin{remark}
As mentioned above, our method is not directly dependent on the slow convergence rate of Monte Carlo methods. We will see that numerical
experiments in Sect. \ref{APP:s4} provides practical evidence for this claim, for example, $n=95$ could work well in $500$ dimensions, and this is obviously inconsistent with practical experience of Monte Carlo methods.
\end{remark}
\begin{proof}
For any $C>0$, since $\rho_\lambda<1$, we have
\begin{equation*}
  n\geqslant\frac{C^2(L+\lambda)^d}
  {2^{\frac{d}{2}}\lambda^{\frac{d}{2}}L^{\frac{d}{2}}}
  \exp\bigg(\frac{\lambda M}{2}\bigg)
  \max\left\{1,\frac{2\rho_\lambda l^2}{5\lambda^2}\right\},
\end{equation*}
and according to Lemma \ref{APP:lem:EV}, we obtain
\begin{equation*}
  n\geqslant\frac{4C^2\mathbb{V}[\psi_k]}{(\mathbb{E}[\psi_k])^2}
  ~~~~\textrm{and}~~~~
  \frac{C^2\mathbb{V}[\phi_k]}{(\mathbb{E}[\psi_k])^2}
  \frac{1}{n}\leqslant\rho^k\frac{\rho_\lambda}{20}\frac{ld}{\lambda^2}
  +\frac{\rho_\lambda}{40}\|x_k-x_*\|_2^2.
\end{equation*}
Together with Lemma \ref{APP:lem:N2I}, we further obtain
\begin{equation*}
  \|x_{k+1}-x_*\|_2^2\leqslant8U_k^2+
  \rho^k\frac{2\rho_\lambda}{5}\frac{ld}{\lambda^2}
  +\frac{\rho_\lambda}{5}\|x_k-x_*\|_2^2,
\end{equation*}
by noting that Lemma \ref{APP:lem:I2N}, we finally obtain, with probability at least $1-\frac{1}{C^2}$,
\begin{align*}
  \|x_{k+1}-x_*\|_2^2\leqslant&8\left(\rho^k\frac{\rho_\lambda}{10}
  \frac{ld}{\lambda^2}+\frac{\rho_\lambda}{10}\|x_k-x_*\|_2^2\right)+
  \rho^k\frac{2\rho_\lambda}{5}\frac{ld}{\lambda^2}
  +\frac{\rho_\lambda}{5}\|x_k-x_*\|_2^2 \\
  \leqslant&\rho^k\rho_\lambda
  \frac{6ld}{5\lambda^2}+\rho_\lambda\|x_k-x_*\|_2^2,
\end{align*}
and the proof is complete.\qed
\end{proof}

The following theorem states $\|x_{k+1}-x_*\|_2^2=\mathcal{O}(\rho^{k+1})$ for all $k\in\mathbb{N}$ in probability. It seems that the convergence of the algorithm implies that each iteration needs to be contracted, i.e., after $k$ iterations, the probability of convergence is $(1-\frac{1}{C^2})^k$, which rapidly tends to $0$ as $k$ increases. However, the APP algorithm actually converges with probability $1$ by using the appropriate number of repetitions with a same variance, as we will see in Theorem \ref{APP:thm:main3}.
\begin{theorem}[A direct extension of Theorem \ref{APP:thm:main1}]
\label{APP:thm:main2}
Under Assumption \ref{APP:ass:A}, suppose that the APP method (Algorithm \ref{APP:alg:APP}) is run with a regularization parameter $\lambda>0$, a contraction factor $0<\rho<1$ and $n\in\mathbb{N}$ such that
\begin{equation*}
  \rho_\lambda=\frac{10\lambda^2L^{\frac{d}{2}}}{l^{\frac{d}{2}+2}} \exp\left(\frac{\lambda^2M}{2l}\right)<\rho<1
\end{equation*}
and
\begin{equation*}
  n\geqslant\frac{C^2(L+\lambda)^d}
  {2^{\frac{d}{2}}\lambda^{\frac{d}{2}}L^{\frac{d}{2}}}
  \exp\bigg(\frac{\lambda M}{2}\bigg)
  \max\left\{1,\frac{2l^2}{5\lambda^2}\right\}.
\end{equation*}
Then in probability, the iterates of APP satisfy for all $k\in\mathbb{N}$:
\begin{equation*}
  \|x_{k+1}-x_*\|_2^2\leqslant
  \rho^k\frac{\rho_\lambda^2}{\rho-\rho_\lambda}
  \frac{6ld}{5\lambda^2}+\rho_\lambda^k\|x_1-x_*\|_2^2
  \leqslant\rho^{k+1}M,
\end{equation*}
where
\begin{equation*}
  M=\frac{6ld}{5\lambda^2}
  \max\left\{1,\frac{\rho_\lambda}{\rho-\rho_\lambda}\right\} +\rho^{-1}\|x_1-x_*\|_2^2.
\end{equation*}
\end{theorem}
\begin{proof}
First, it follows from the definition of $M$ that $\|x_1-x_*\|_2^2\leqslant\rho M$; therefore, according to Theorem \ref{APP:thm:main2}, it follows that
\begin{equation*}
  \|x_2-x_*\|_2^2\leqslant
  \rho\rho_\lambda\frac{6ld}{5\lambda^2}+\rho_\lambda\|x_1-x_*\|_2^2
\end{equation*}
in probability; meanwhile, since $\rho_\lambda<\rho$, we further have
\begin{equation*}
  \|x_2-x_*\|_2^2\leqslant\rho^2\left(\frac{6ld}{5\lambda^2}
  +\rho^{-1}\|x_1-x_*\|_2^2\right)\leqslant\rho^2M.
\end{equation*}
And similarly, we could get
\begin{equation*}
  \|x_3-x_*\|_2^2\leqslant\rho^2\rho_\lambda
  \frac{6ld}{5\lambda^2}\left(1+\frac{\rho_\lambda}{\rho}\right)
  +\rho_\lambda^2\|x_1-x_*\|_2^2\leqslant\rho^3M.
\end{equation*}
Doing it recursively, one obtain that, in probability, for all $k\in\mathbb{N}$, the iterates of APP satisfy
\begin{align*}
  \|x_{k+1}-x_*\|_2^2\leqslant&\rho^k\rho_\lambda
  \frac{6ld}{5\lambda^2}\left(1+\frac{\rho_\lambda}{\rho}+
  \cdots+\frac{\rho_\lambda^k}{\rho^k}\right)
  +\rho_\lambda^k\|x_1-x_*\|_2^2 \\
  \leqslant&\rho^k\frac{\rho_\lambda^2}{\rho-\rho_\lambda}
  \frac{6ld}{5\lambda^2}+\rho_\lambda^k\|x_1-x_*\|_2^2
  \leqslant\rho^{k+1}M,
\end{align*}
and the proof is complete.\qed
\end{proof}

The following conclusion shows that the APP iterates linearly converges to $x_*$, almost with probability one.
\begin{theorem}
\label{APP:thm:main3}
Under Assumption \ref{APP:ass:A}, suppose that the APP method (Algorithm \ref{APP:alg:APP}) is run with a regularization parameter $\lambda>0$, a contraction factor $0<\rho<1$, $n\in\mathbb{N}$ and $\gamma>1$ such that
\begin{equation*}
  \rho_\lambda=\frac{10\lambda^2L^{\frac{d}{2}}}{l^{\frac{d}{2}+2}} \exp\left(\frac{\lambda^2M'}{2l}\right)<\frac{\rho}{\gamma}
\end{equation*}
and
\begin{equation*}
  n\geqslant\frac{C^2(L+\lambda)^d}
  {2^{\frac{d}{2}}\lambda^{\frac{d}{2}}L^{\frac{d}{2}}}
  \exp\bigg(\frac{\lambda M'}{2}\bigg)
  \max\left\{1,\frac{2l^2}{5\lambda^2}\right\}.
\end{equation*}
If at least one of every $s$ successive iterates, say, $\{x_{i+j}\}_{j=1}^s$, satisfies the contraction inequality \eqref{APP:eq:main1} and $\|x_{i+j}-x_*\|_2^2\leqslant \gamma\|x_i-x_*\|_2^2$ for every $0\leqslant j\leqslant s$, and the variance is reduced once after $s$ iterations, then the iterates of APP satisfy for all $k\in\mathbb{N}$:
\begin{align*}
  \|x_{ks}-x_*\|_2^2\leqslant\frac{\rho^{k+1}}{\gamma-1}
  \frac{6ld}{5\lambda^2}+\rho^k\|x_1-x_*\|_2^2,
\end{align*}
and for every $0\leqslant j\leqslant s$,
\begin{align*}
  \|x_{ks+j}-x_*\|_2^2\leqslant\rho^{k+1}M',
\end{align*}
where
\begin{equation*}
  M'=\frac{6ld}{5\lambda^2}\frac{\gamma}{\gamma-1}
  +\frac{\gamma}{\rho}\|x_1-x_*\|_2^2.
\end{equation*}
\end{theorem}
\begin{proof}
Without loss of generality, assume that for every $k\in\mathbb{N}_0$, in each $s$ successive iterates $\{x_{ks+j}\}_{j=1}^s$, only the last iteration $x_{(k+1)s}$ satisfies the contraction inequality \eqref{APP:eq:main1}, then
\begin{equation*}
  \|x_{s-1}-x_*\|_2^2\leqslant \gamma\|x_1-x_*\|_2^2\leqslant\rho M',
\end{equation*}
therefore, according to the contraction inequality \eqref{APP:eq:main1}, it follows that
\begin{equation*}
  \|x_s-x_*\|_2^2\leqslant
  \rho\rho_\lambda\frac{6ld}{5\lambda^2}+\rho_\lambda\|x_{s-1}-x_*\|_2^2
  <\frac{\rho^2}{\gamma}\frac{6ld}{5\lambda^2}+\rho\|x_1-x_*\|_2^2,
\end{equation*}
meanwhile, for every $1\leqslant j\leqslant s$, we also have
\begin{equation*}
  \|x_{s+j}-x_*\|_2^2\leqslant\gamma\|x_s-x_*\|_2^2
  \leqslant\rho^2\left(\frac{6ld}{5\lambda^2}
  +\frac{\gamma}{\rho}\|x_1-x_*\|_2^2\right)\leqslant\rho^2M'.
\end{equation*}
And similarly, we could get
\begin{equation*}
  \|x_{2s}-x_*\|_2^2\leqslant\frac{\rho^3}{\gamma}
  \frac{6ld}{5\lambda^2}\left(1+\frac{1}{\gamma}\right)
  +\rho^2\|x_1-x_*\|_2^2,
\end{equation*}
and for every $1\leqslant j\leqslant s$,
\begin{equation*}
  \|x_{2s+j}-x_*\|_2^2\leqslant\gamma\|x_s-x_*\|_2^2
  \leqslant\rho^3\left[\frac{6ld}{5\lambda^2}\left(1+\frac{1}{\gamma}\right)
  +\frac{\gamma}{\rho}\|x_1-x_*\|_2^2\right]\leqslant\rho^3M'.
\end{equation*}
Doing it recursively, one obtain that, for all $k\in\mathbb{N}$, the iterates of APP satisfy
\begin{align*}
  \|x_{ks}-x_*\|_2^2\leqslant&\frac{\rho^{k+1}}{\gamma}
  \frac{6ld}{5\lambda^2}\left(1+\frac{1}{\gamma}+
  \cdots+\frac{1}{\gamma^k}\right)+\rho^k\|x_1-x_*\|_2^2 \\
  \leqslant&\frac{\rho^{k+1}}{\gamma-1}
  \frac{6ld}{5\lambda^2}+\rho^k\|x_1-x_*\|_2^2,
\end{align*}
and for every $0\leqslant j\leqslant s$,
\begin{align*}
  \|x_{ks+j}-x_*\|_2^2\leqslant\rho^{k+1}M',
\end{align*}
and the proof is complete.\qed
\end{proof}

Since the number of function evaluations per-iteration, i.e., $n$, is a fixed number independent of $k$, the following corollary is immediate from Theorem \ref{APP:thm:main2}. It provides a total work complexity bound for the APP methods.
\begin{corollary}[Complexity bound]
Suppose the conditions of Theorem \ref{APP:thm:main2} hold. Then the
number of function evaluations of an APP (Algorithm \ref{APP:alg:APP}) required to achieve $\|x_k-x_*\|_2^2\leqslant\epsilon$ is $\mathcal{O}(\log(1/\epsilon))$.
\end{corollary}

\subsection{Stable algorithm}

Algorithm \ref{APP:alg:APP} has a hidden danger that the increase of $k$ might cause underflow. For example, assume the $f(\theta)\geqslant1$ for every $\theta\in\mathbb{R}^d$. If $\rho=0.9$, then
\begin{equation*}
  \exp(-\rho^{-k}f(\theta))\leqslant\exp(-\rho^{-k})=\exp(-0.9^{-k}).
\end{equation*}
And it is worth noting that
\begin{equation*}
  \exp(-0.9^{-62.29})=2.38389\times10^{-308}.
\end{equation*}
According to the IEEE $754$ double-precision floating-point standard, $\exp(-0.9^{-k})$ will be always stored as $0$ when $k\geqslant63$. Recall the APP iteration
\begin{equation*}
  x_{k+1}=\frac{\sum_{i=1}^n\theta_i \exp(-\rho^{-k}f(\theta_i))}{\sum_{i=1}^n \exp(-\rho^{-k}f(\theta_i))},
\end{equation*}
the iteration will terminate abnormally when $k=63$, because the denominator is zero. However, the current contraction ratio $0.9^{62}\approx0.0015$ is obviously not enough for most practical problems. To address this issue, we introduce the stable APP method as Algorithm \ref{APP:alg:SAPP} below. For this version, the second moment of function value sequence is used to avoid such underflows in each iteration.

\begin{algorithm}
\caption{Stable APP Method}
\label{APP:alg:SAPP}
\begin{algorithmic}[1]
\STATE{Choose an initial iterate $x_1$ and preset parameters $\lambda>0$, $\rho\in(0,1)$, $n\in\mathbb{N}$.}
\STATE{Set the current best $f_*^c=\infty$.}
\FOR{$k=1,2,\cdots$}
\STATE{Generate $n$ realizations $\{\theta_i\}_i^n$ of the random vector from $\mathcal{N}(x_k,\rho^{k}\lambda^{-1}I_d)$.}
\STATE{Compute function value sequence $\{f(\theta_i)\}_i^n$.}
\STATE{Update the current best $f_*^c=\min\{f_*^c,\min_{1\leqslant i\leqslant n}f(\theta_i)\}$.}
\STATE{Set $y_i=f(\theta_i)-f_*^c$ for $1\leqslant i\leqslant n$ and compute $\hat{m}_k^2=\frac{1}{n}\sum_{i=1}^ny_i^2$.}
\STATE{Set the new iterate as $x_{k+1}=\frac{\sum_{i=1}^n\theta_i \exp(-\hat{m}_k^{-1}y_i)}{\sum_{i=1}^n\exp(-\hat{m}_k^{-1}y_i)}$.}
\ENDFOR
\end{algorithmic}
\end{algorithm}

Let $m_k^2=\mathbb{E}[(f(\theta_i)-f_*)^2]$, then the stable APP iteration can be approximately written as
\begin{equation*}
  x_{k+1}=\frac{\sum_{i=1}^n\theta_i \exp[-m_k^{-1}(f(\theta_i)-f_*)]}{\sum_{i=1}^n \exp[-m_k^{-1}(f(\theta_i)-f_*)]},~~~
  \theta_i\sim\mathcal{N}(x_k,\rho^{k}\lambda^{-1}I_d),
\end{equation*}
which is an estimate for
\begin{equation*}
  e_{k+1}=\frac{\int_{\mathbb{R}^d}x~\exp\big[-\big(m_k^{-1}
  (f(x)-f_*)+\frac{\rho^{-k}\lambda}{2}\|x-x_k\|_2^2\big)\big]\ud x}
  {\int_{\mathbb{R}^d}\exp\big[-\big(m_k^{-1}(f(x)-f_*)
  +\frac{\rho^{-k}\lambda}{2}\|x-x_k\|_2^2\big)\big]\ud x}.
\end{equation*}
Notice that $\mathbb{V}[m_k^{-1}(f(\theta)-f_*)]=1$, the underflow mentioned above does not occur in this stable APP Method. As shown in Algorithm \ref{APP:alg:SAPP}, $f_*$ and $m_k^2$ can be replaced with corresponding estimates $f_*^c$ and $\hat{m}_k^2$, respectively.

We will see that, Theorem \ref{APP:thm:mk} establishes upper and lower bounds for $m_k^{-1}$, i.e., there are $0<c\leqslant C<\infty$ such that $c\rho^{-k}\leqslant m_k^{-1}\leqslant C\rho^{-k}$. And this yields
\begin{align*}
  e_{k+1}\leqslant\frac{\int_{\mathbb{R}^d}x~\exp\big[-\rho^{-k}
  \big(\frac{cl}{2}\|x-x_k\|_2^2
  +\frac{\lambda}{2}\|x-x_k\|_2^2\big)\big]\ud x}
  {\int_{\mathbb{R}^d}\exp\big[-\rho^{-k}
  \big(\frac{CL}{2}\|x-x_k\|_2^2
  +\frac{\lambda}{2}\|x-x_k\|_2^2\big)\big]\ud x},
\end{align*}
therefore, by making the substitution $l'=cl$ and $L'=CL$, Jensen's inequality gives
\begin{equation*}
  \|e_{k+1}-x_*\|\leqslant\frac{\int_{\mathbb{R}^d}
  \|x-x_*\|_2^2~\exp\big[-\rho^{-k}
  \big(\frac{l'}{2}\|x-x_k\|_2^2
  +\frac{\lambda}{2}\|x-x_k\|_2^2\big)\big]\ud x}
  {\int_{\mathbb{R}^d}\exp\big[-\rho^{-k}
  \big(\frac{L'}{2}\|x-x_k\|_2^2
  +\frac{\lambda}{2}\|x-x_k\|_2^2\big)\big]\ud x}.
\end{equation*}
Hence, recall \eqref{APP:eq:IkRB} and \eqref{APP:eq:I2Nt4}, the analyses of Algorithm \ref{APP:alg:APP} can be directly applied to Algorithm \ref{APP:alg:SAPP}.

\begin{theorem}
\label{APP:thm:mk}
Under Assumption \ref{APP:ass:A}, suppose that the stable APP method (Algorithm \ref{APP:alg:SAPP}) is run with a regularization parameter $\lambda>0$, a contraction factor $0<\rho<1$ and $n\in\mathbb{N}$ such that $\|x_k-x_*\|_2^2\leqslant\rho^kM$ for a fixed $M>0$ and $k\in\mathbb{N}$. Let
\begin{equation*}
  m_k^2=\left(\frac{\rho^{-k}\lambda}{2\pi}\right)^{\frac{d}{2}}
  \int_{\mathbb{R}^d}\big(f(x)-f_*\big)^2\exp\Big(-
  \frac{\rho^{-k}\lambda}{2}\|x-x_k\|_2^2\Big)\ud x,
\end{equation*}
then $m_k^{-1}$ has the following upper and lower bounds:
\begin{equation*}
  c\rho^{-k}\leqslant m_k^{-1}\leqslant C\rho^{-k},
\end{equation*}
where $c=\frac{2\lambda}{Ld\sqrt{3+6M\lambda+M^2\lambda^2}}$ and $C=\frac{2\lambda}{l\sqrt{3d}}$.
\end{theorem}
\begin{proof}
From Assumption \ref{APP:ass:A}, we have
\begin{equation*}
  \frac{l^2}{4}\|x-x_*\|_2^4\leqslant
  \big(f(x)-f_*\big)^2\leqslant\frac{L^2}{4}\|x-x_*\|_2^4,
\end{equation*}
together with the definition of $m_k^2$, i.e., $m_k^2=\mathbb{E}[(f(x)-f_*)^2]$, it follows that
\begin{equation*}
  \frac{l^2}{4}\mathbb{E}\big[\|x-x_*\|_2^4\big]\leqslant
  m_k^2\leqslant\frac{L^2}{4}\mathbb{E}\big[\|x-x_*\|_2^4\big],
  ~~\textrm{where}~~x\sim\mathcal{N}(x_k,\rho^{k}\lambda^{-1}I_d).
\end{equation*}
Let $a^{(i)}$ be the $i$th component of a vector $a\in\mathbb{R}^d$. Notice that
\begin{equation*}
  \|x-x_*\|_2^4=\bigg(\sum_{i=1}^d
  \big(x^{(i)}-x_*^{(i)}\big)^2\bigg)^2,
\end{equation*}
we obtain
\begin{equation*}
  \|x-x_*\|_2^4=\sum_{i=1}^d
  \big(x^{(i)}-x_*^{(i)}\big)^4+\sum_{i\neq j}
  \big(x^{(i)}-x_*^{(i)}\big)^2\big(x^{(j)}-x_*^{(j)}\big)^2
  \geqslant\sum_{i=1}^d\big(x^{(i)}-x_*^{(i)}\big)^4
\end{equation*}
and
\begin{equation*}
  \|x-x_*\|_2^4\leqslant d\sum_{i=1}^d\big(x^{(i)}-x_*^{(i)}\big)^4,
\end{equation*}
the last inequality comes directly from the famous Root Mean Square Arithmetic Mean inequality; therefore, we obtain
\begin{equation}\label{APP:eq:mkA}
  \frac{l^2}{4}\sum_{i=1}^d\mathbb{E}
  \Big[\big(x^{(i)}-x_*^{(i)}\big)^4\Big]
  \leqslant m_k^2\leqslant
  \frac{L^2d}{4}\sum_{i=1}^d\mathbb{E}
  \Big[\big(x^{(i)}-x_*^{(i)}\big)^4\Big].
\end{equation}

Further, by noting that
\begin{align*}
  &\mathbb{E}\big(x^{(i)}-x_*^{(i)}\big)^4 \\
  =&\left(\frac{\rho^{-k}\lambda}{2\pi}\right)^{\frac{d}{2}}
  \int_{\mathbb{R}^d}\big(x^{(i)}-x_*^{(i)}\big)^4\exp
  \Big(-\frac{\rho^{-k}\lambda}{2}\|x-x_k\|_2^2\Big)\ud x \\
  =&\left(\frac{\rho^{-k}\lambda}{2\pi}\right)^{\frac{1}{2}}
  \int_{\mathbb{R}^d}\Big[\big(x^{(i)}-x_k^{(i)}\big)+
  \big(x_k^{(i)}-x_*^{(i)}\big)\Big]^4
  \exp\Big(-\frac{\rho^{-k}\lambda}{2}
  \|x^{(i)}-x_k^{(i)}\|_2^2\Big)\ud x^{(i)} \\
  =&\left(\frac{3\rho^{2k}}{\lambda^2}
  +\frac{6\rho^k}{\lambda}\big(x_k^{(i)}-x_*^{(i)}\big)^2
  +\big(x_k^{(i)}-x_*^{(i)}\big)^4\right)
\end{align*}
and
\begin{equation*}
  0\leqslant\big(x_k^{(i)}-x_*^{(i)}\big)^2\leqslant
  \|x_k-x_*\|_2^2\leqslant\rho^kM,
\end{equation*}
it follows that
\begin{equation}\label{APP:eq:mkB}
  \frac{3\rho^{2k}}{\lambda^2}\leqslant
  \mathbb{E}\big(x^{(i)}-x_*^{(i)}\big)^4\leqslant
  \frac{(3+6M\lambda+M^2\lambda^2)\rho^{2k}}{\lambda^2}.
\end{equation}

Finally, it follows from \eqref{APP:eq:mkA} and \eqref{APP:eq:mkB} that
\begin{equation*}
  \frac{3l^2d}{4\lambda^2}\rho^{2k}\leqslant m_k^2\leqslant
  \frac{L^2d^2(3+6M\lambda+M^2\lambda^2)}{4\lambda^2}\rho^{2k},
\end{equation*}
and the proof is complete. \qed
\end{proof}

\section{Numerical experiments}
\label{APP:s4}

Now we illustrate the numerical performance of the algorithm. We mainly consider the revised Rastrigin function in $\mathbb{R}^d$ defined as
\begin{equation*}
  f(x)=\log\bigg(\|x\|_2^2-\frac{1}{2}\sum_{i=1}^d\cos\big(5\pi x^{(i)}\big)
  +\frac{d}{2}+\frac{1}{10}\bigg)-\log\bigg(\frac{1}{10}\bigg),
\end{equation*}
where $x^{(i)}$ be the $i$th component of $x$. As mentioned before, it is very similar to the funnel-shaped function shown in Fig. \ref{APP:fig:PF} (left) and satisfies Assumption \ref{APP:ass:A} within an appropriate range. It is worth noting that, although $\exp(f)$ is separable, we obviously did not make use of its separability, because the actual folding energy landscape may not be separable.

Clearly, the function above has a unique global minima located at the origin and huge local minima. Only in the hypercube $[-1,1]^d$, the number of its local minima reaches $5^d$, e.g., about $3.055\times10^{349}$ for $d=500$. This number is roughly the same as the number of possible conformations of a protein composed of hundreds of amino acids. In the following experiments, every initial iterate is selected on a $d$-dimensional sphere of radius $\sqrt{d}$ centered at the origin. Therefore, finding the global minima is extremely difficult for a large $d$.

\subsection{Performances of APP in various dimensions from $2$ to $500$}

In experiments below, the random vectors in each iteration are generated by a halton sequence with RR$2$ scramble type \cite{KocisL1997A_QMCscramble}. And each point in the sequence will only be used once, that is, each point will not be used repeatedly.

Fig. \ref{APP:fig:1} shows the intuitive convergence behavior of the APP algorithm for the revised Rastrigin function in $2$ dimensions. One can see that the APP algorithm not only guarantees global linear convergence, but also tends to converge directly towards the global minimizer, despite the existence of numerous local minima.

\begin{figure}[tbhp]
\centering
\subfigure{\includegraphics[width=0.325\textwidth]{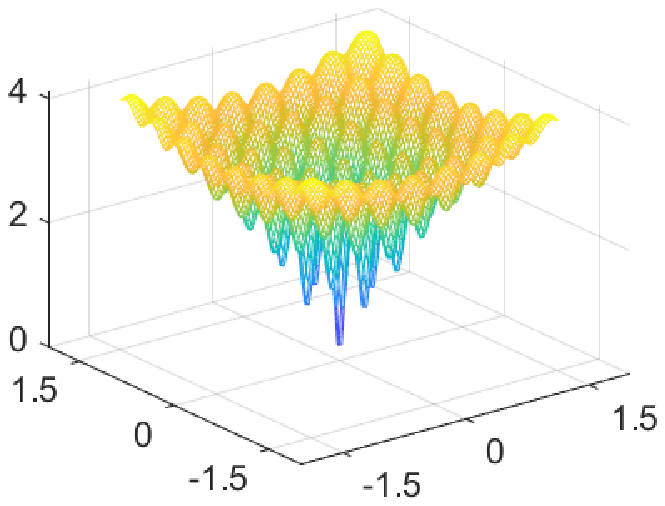}}
\subfigure{\includegraphics[width=0.325\textwidth]{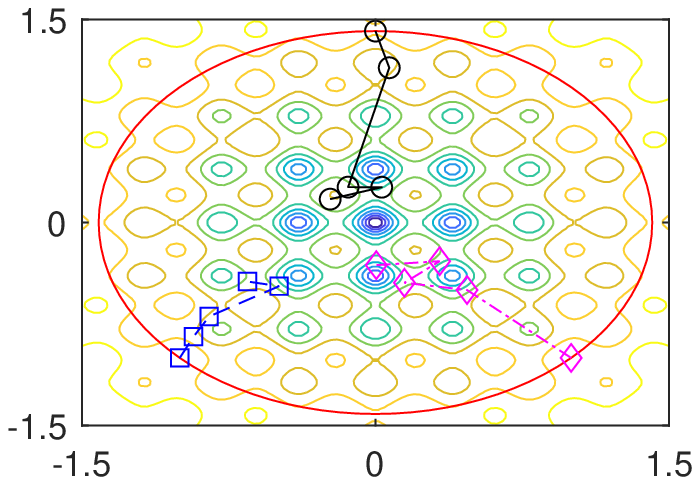}}
\subfigure{\includegraphics[width=0.325\textwidth]{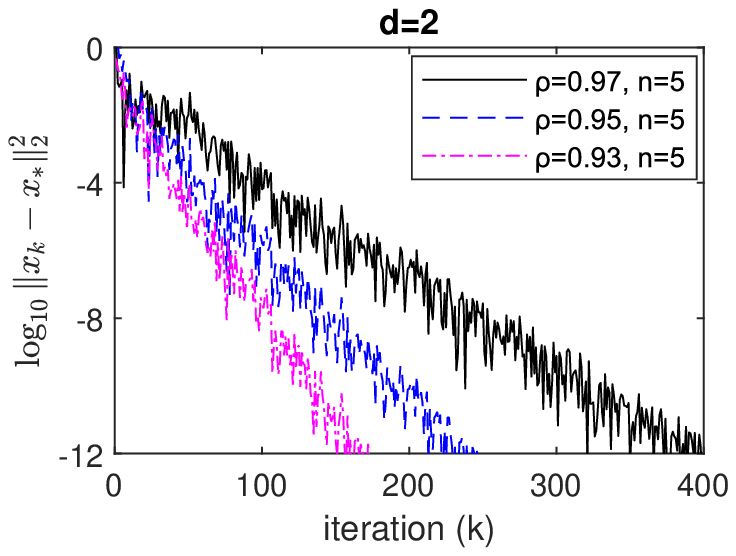}}
\caption{Performance of APP method for the revised Rastrigin function in $2$ dimensions, three independent initial iterates $(0,\sqrt{2})$ and $(\pm1,-1)$ are selected on a sphere of radius $\sqrt{2}$ centered at the origin, and the parameter $\lambda=1/\sqrt{2}$. Left: the objective function. Middle: the first $5$ iterates for each independent run. Right: global convergence behavior for each independent run with relevant parameter setting.}
\label{APP:fig:1}
\end{figure}

\begin{figure}[tbhp]
\centering
\subfigure{\includegraphics[width=0.325\textwidth]{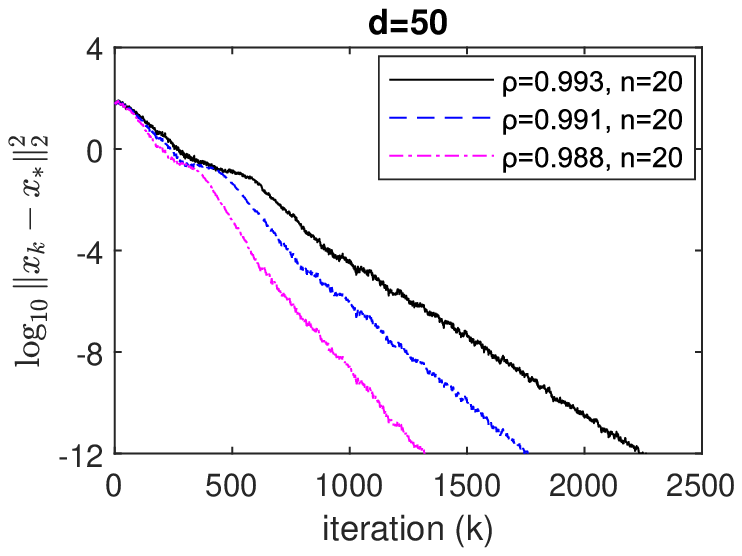}}
\subfigure{\includegraphics[width=0.325\textwidth]{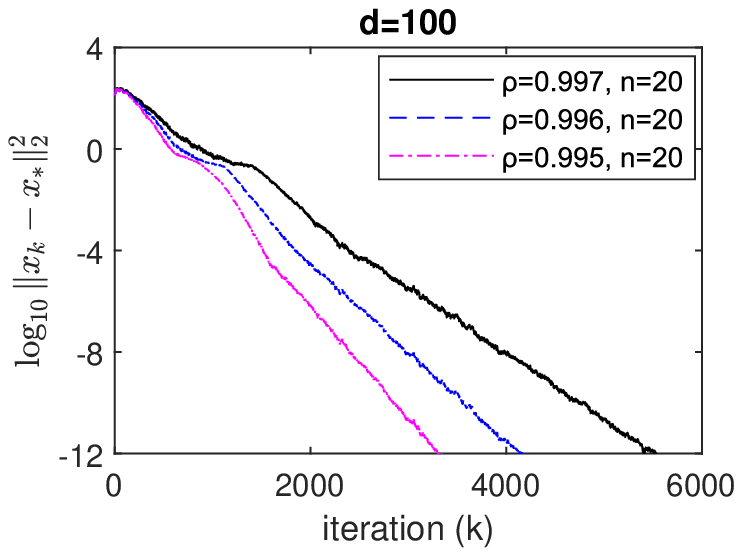}}
\subfigure{\includegraphics[width=0.325\textwidth]{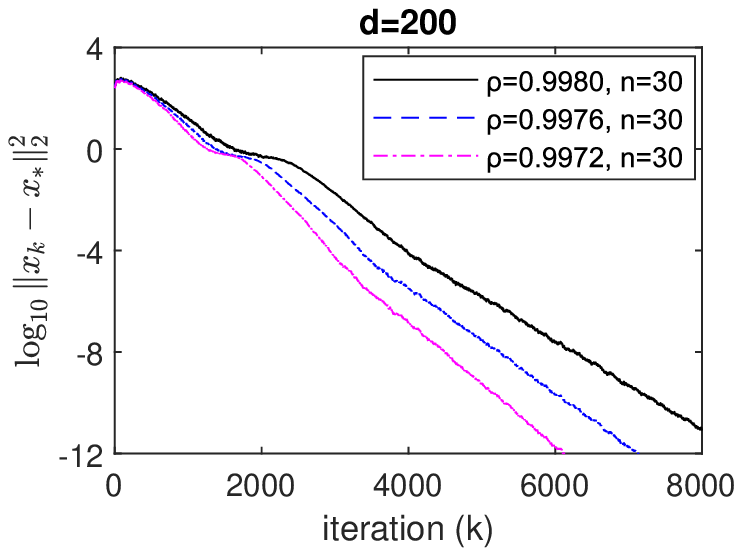}}
\subfigure{\includegraphics[width=0.325\textwidth]{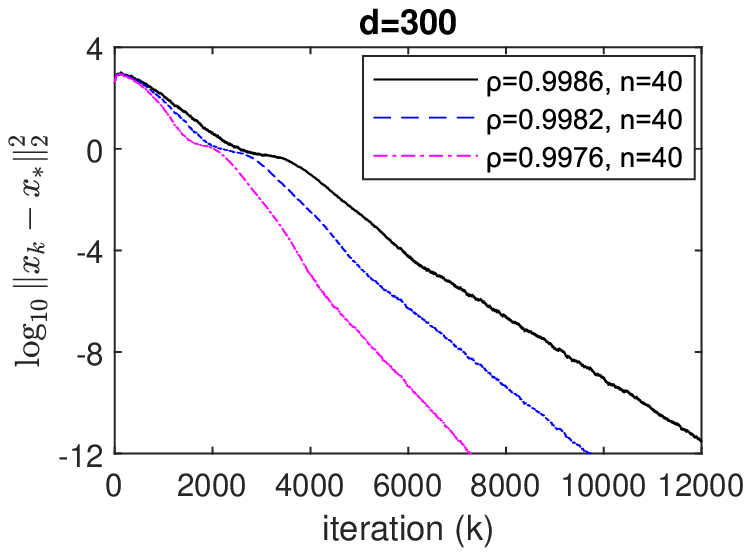}}
\subfigure{\includegraphics[width=0.325\textwidth]{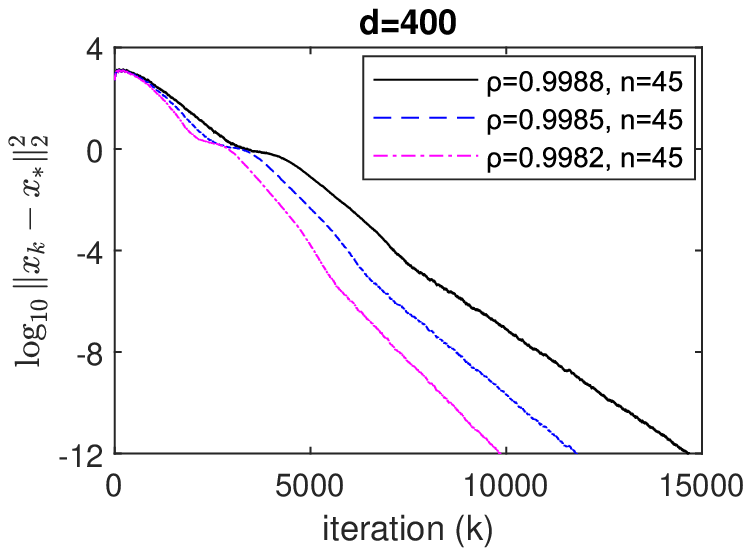}}
\subfigure{\includegraphics[width=0.325\textwidth]{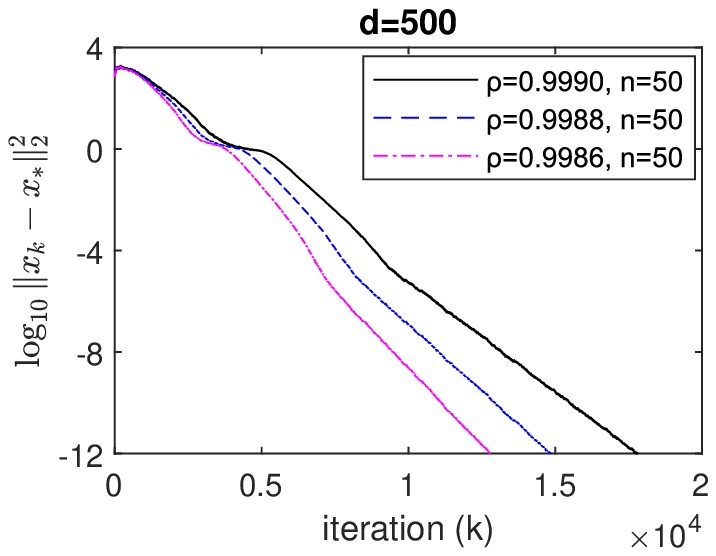}}
\caption{Performances of APP method for the revised Rastrigin function in various dimensions, every initial iterate is randomly selected on a sphere of radius $\sqrt{d}$ centered at the origin, the parameter $\lambda=1/\sqrt{d}$, three different settings for the parameters $\rho$ and $n$ are run independently for each plot.}
\label{APP:fig:2}
\end{figure}

Fig. \ref{APP:fig:2} shows the performance of the APP algorithm in various dimensions from $50$ to $500$. These experiments further demonstrate global linear convergence. The oscillation of error in Fig. \ref{APP:fig:2} may be related to the mismatch between the fixed preset parameters and the local characteristics of the function. Moreover, one does not need a large $n$ to guarantee convergence.

\subsection{Comparisons of APP and DE in various dimensions from $50$ to $500$}

Differential Evolution (DE) is a simple and effective evolutionary
algorithm widely used to solve global optimization problems in a continuous
domain \cite{OparaK2019R_DE}, and it proved to be the fastest evolutionary algorithm among most entries in the First International IEEE Competition on Evolutionary Optimization \cite{StornR1997M_DifferentialEvolution}. There are different variants and we follow the standard convention to classify them by using the notation
\begin{equation*}
  DE/X/Y/Z,
\end{equation*}
where $X$ is the choice of base vector which can typically be ``rand'' or ``best'', $Y$ is the number of difference vectors used, and $Z$ denotes the crossover scheme which can usually be ``bin'' or ``exp''. The most common variant is binomial crossover, denoted as $DE/X/Y/bin$ \cite{OparaK2019R_DE}, and we will use $DE/rand/1/bin$ for our comparisons, because $DE/best/Y/Z$ is not suitable for problems with huge local minima, due to its greedy mechanism.

Comparisons of APP and DE are shown in the left plots of Fig. \ref{APP:fig:3}. The choice of DE's parameters follows the existing experiences \cite{OparaK2019R_DE,StornR1997M_DifferentialEvolution,Yang2014M_Opt} and has been fine-tuned as shown in the middle and right plots of each row of Fig. \ref{APP:fig:3}. These comparisons reflect the advantages of APP in problems satisfied our assumption. 

\begin{figure}[tbhp]
\centering
\subfigure{\includegraphics[width=0.325\textwidth]{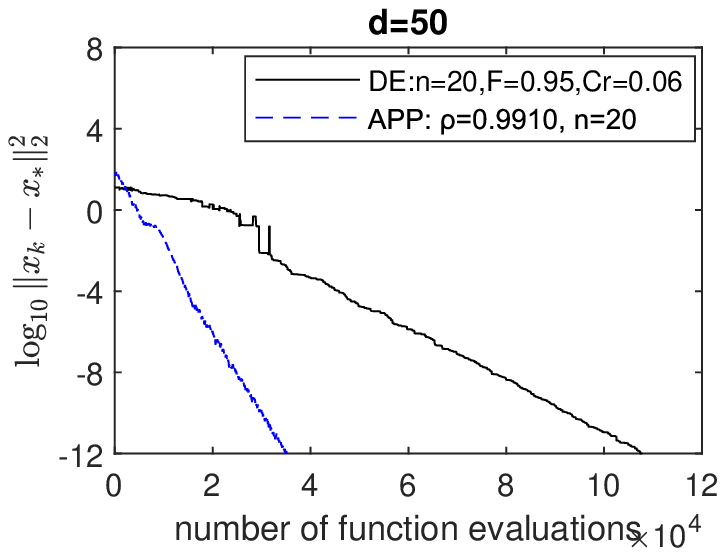}}
\subfigure{\includegraphics[width=0.325\textwidth]{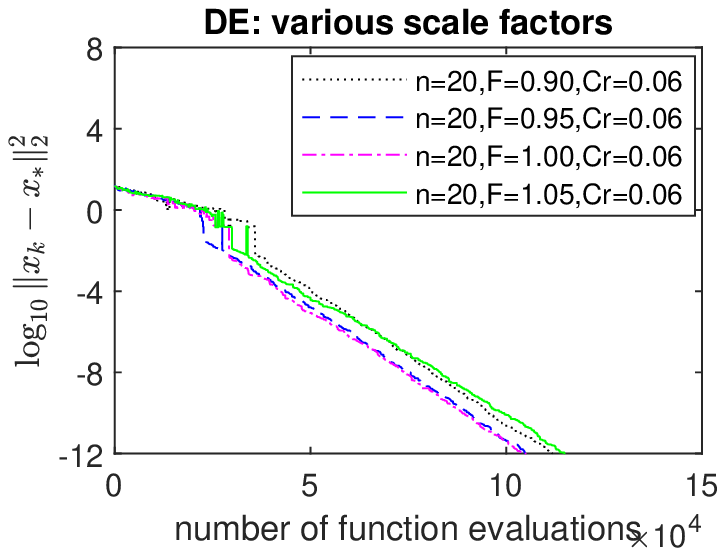}}
\subfigure{\includegraphics[width=0.325\textwidth]{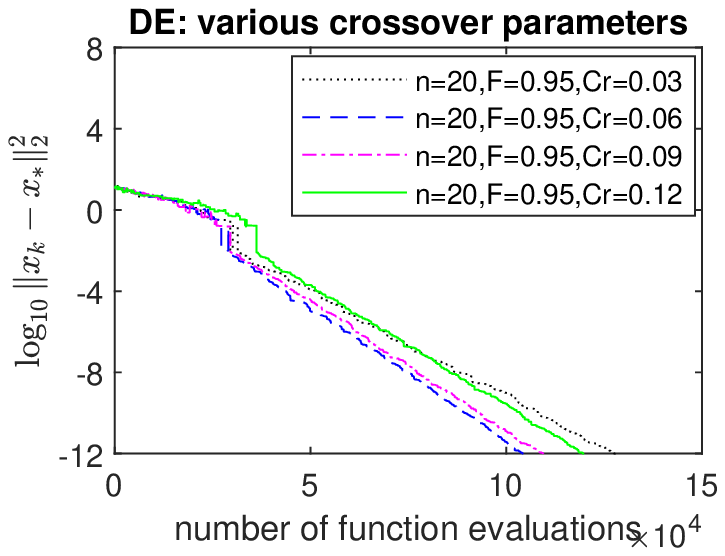}}
\subfigure{\includegraphics[width=0.325\textwidth]{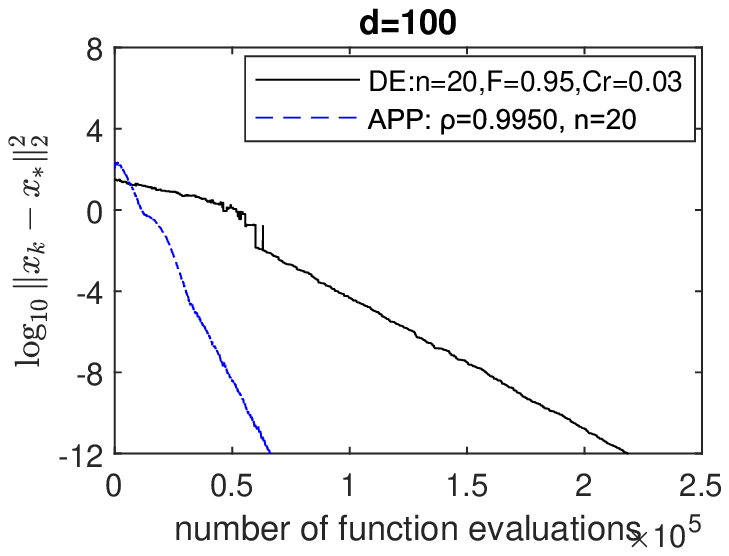}}
\subfigure{\includegraphics[width=0.325\textwidth]{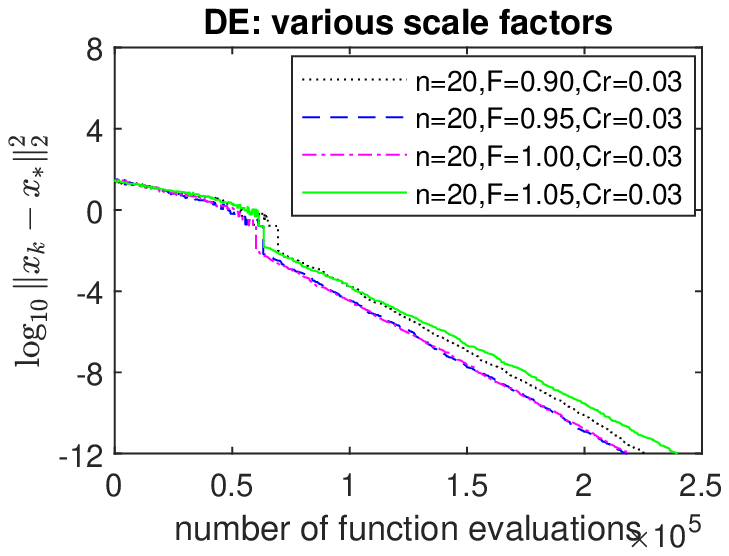}}
\subfigure{\includegraphics[width=0.325\textwidth]{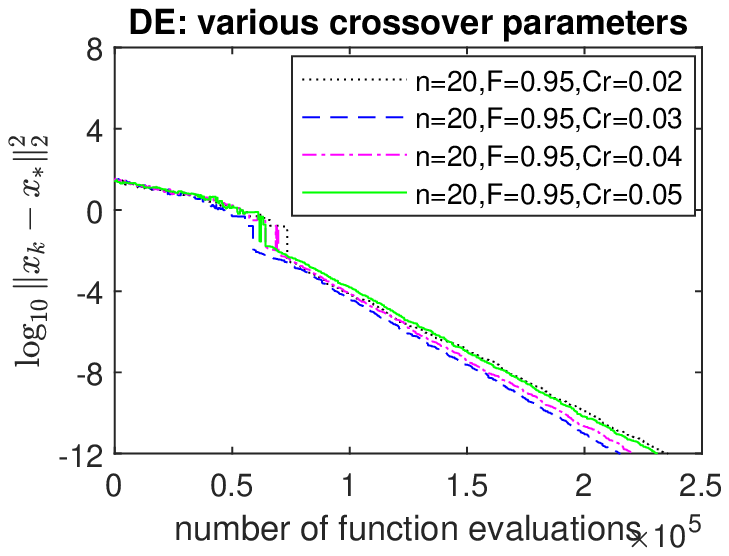}}
\subfigure{\includegraphics[width=0.325\textwidth]{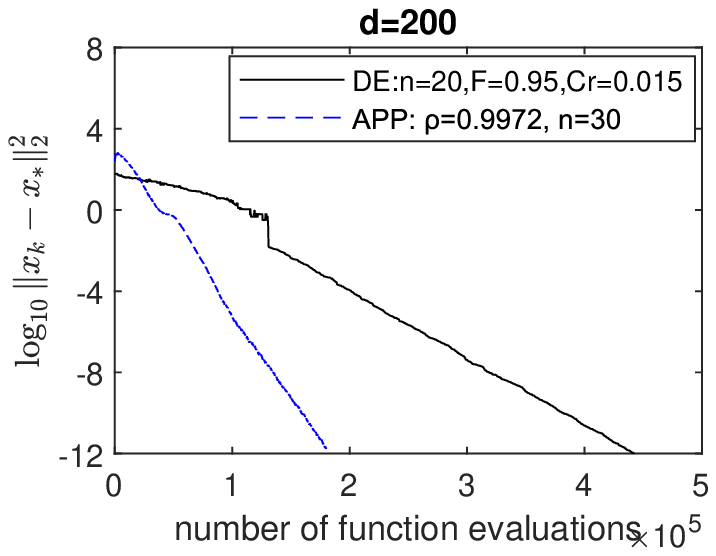}}
\subfigure{\includegraphics[width=0.325\textwidth]{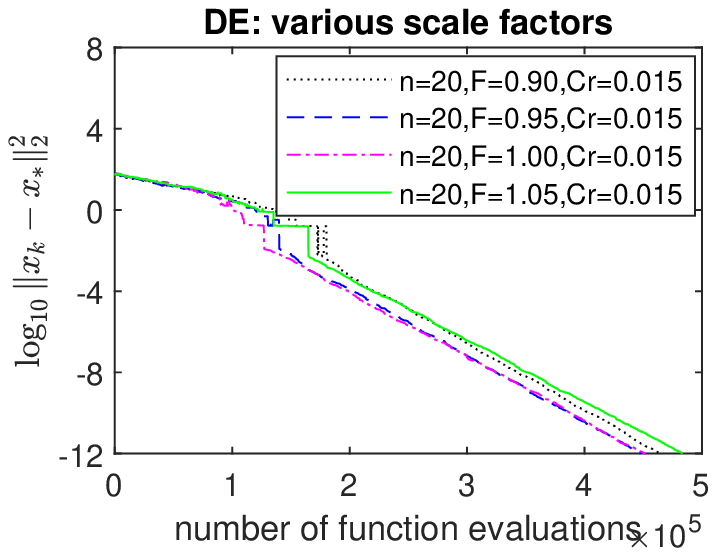}}
\subfigure{\includegraphics[width=0.325\textwidth]{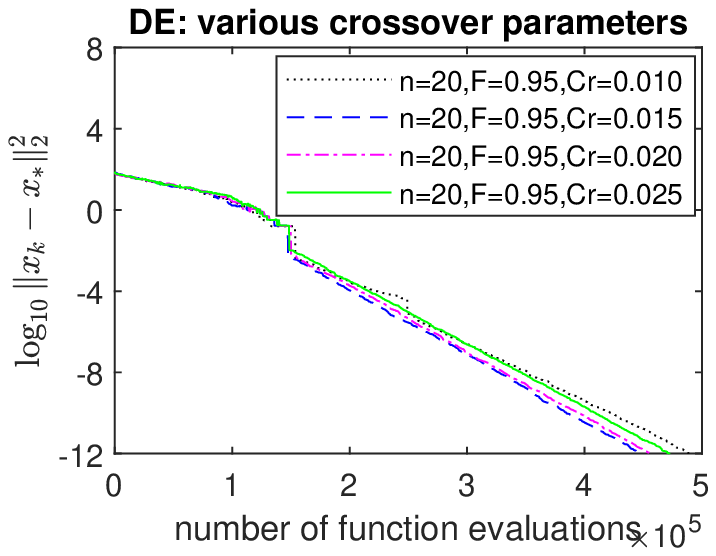}}
\subfigure{\includegraphics[width=0.325\textwidth]{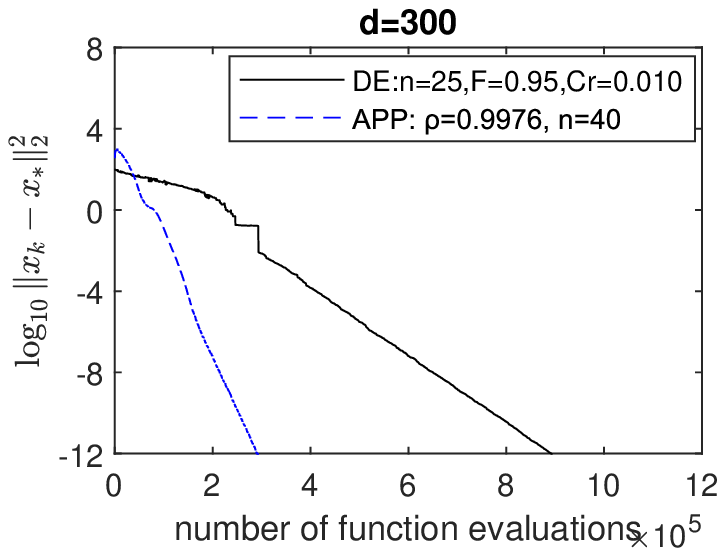}}
\subfigure{\includegraphics[width=0.325\textwidth]{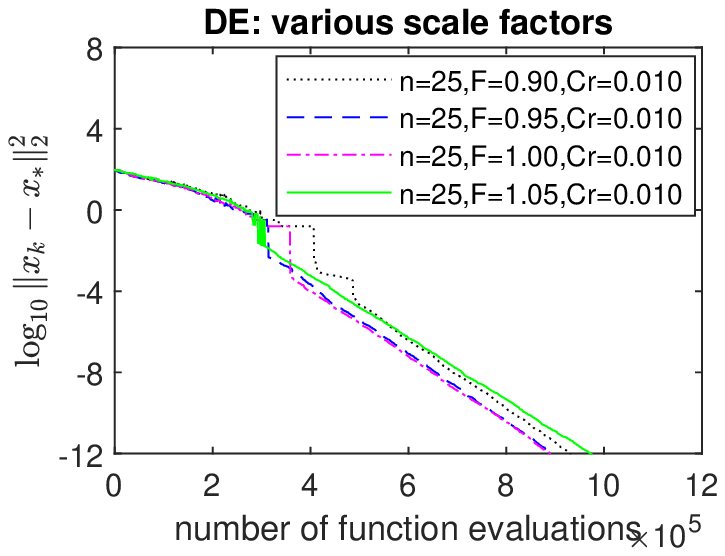}}
\subfigure{\includegraphics[width=0.325\textwidth]{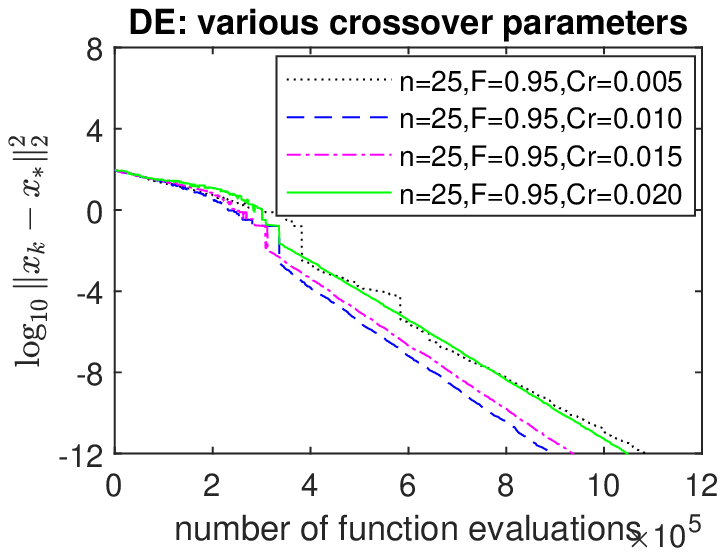}}
\subfigure{\includegraphics[width=0.325\textwidth]{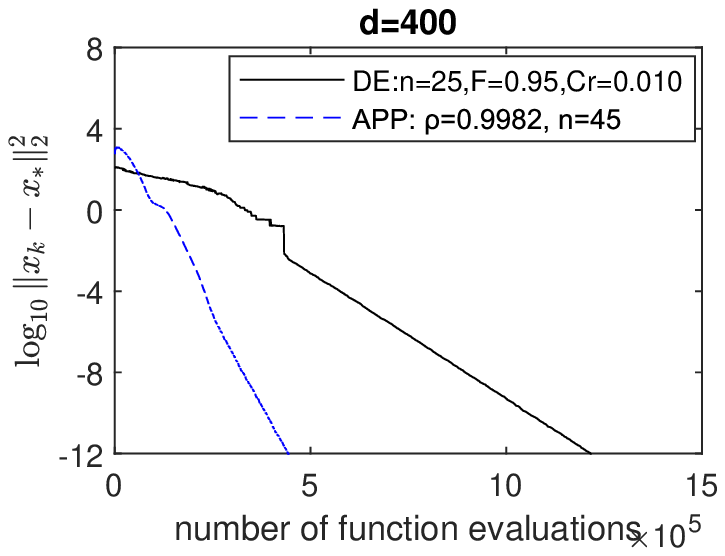}}
\subfigure{\includegraphics[width=0.325\textwidth]{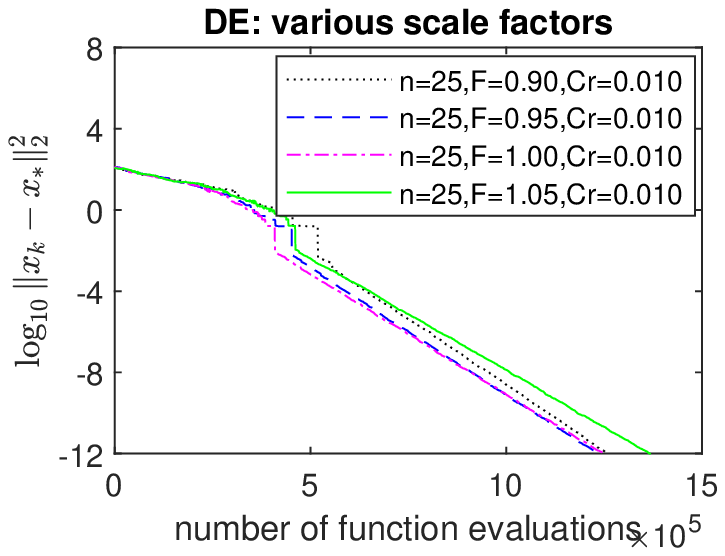}}
\subfigure{\includegraphics[width=0.325\textwidth]{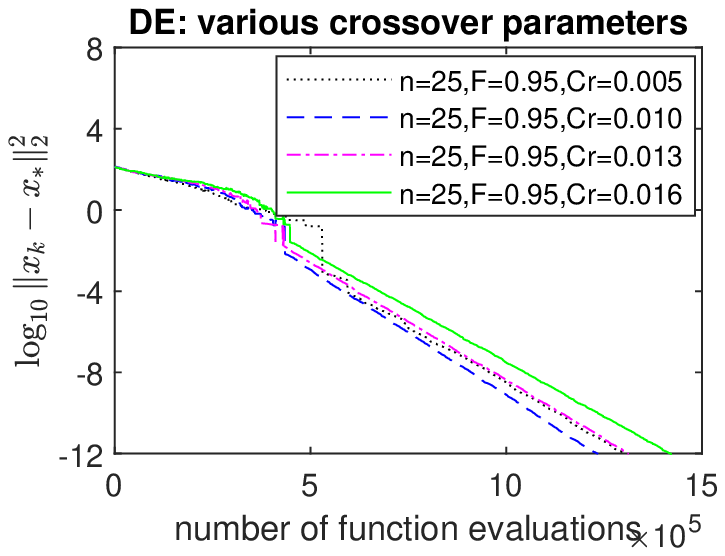}}
\subfigure{\includegraphics[width=0.325\textwidth]{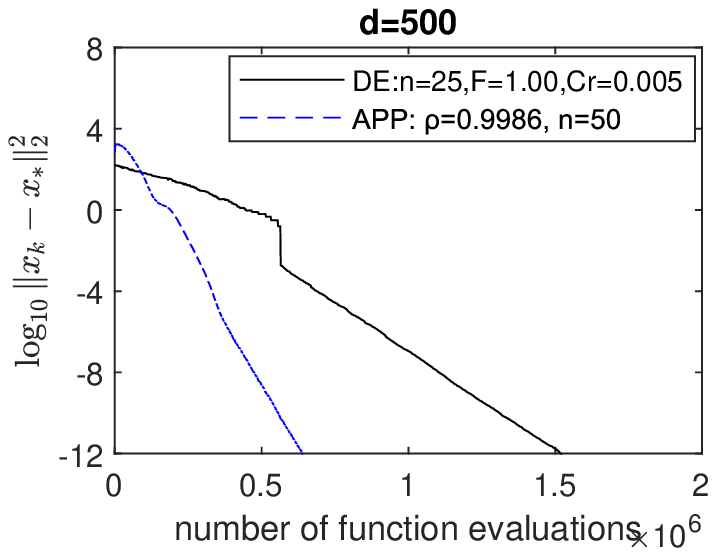}}
\subfigure{\includegraphics[width=0.325\textwidth]{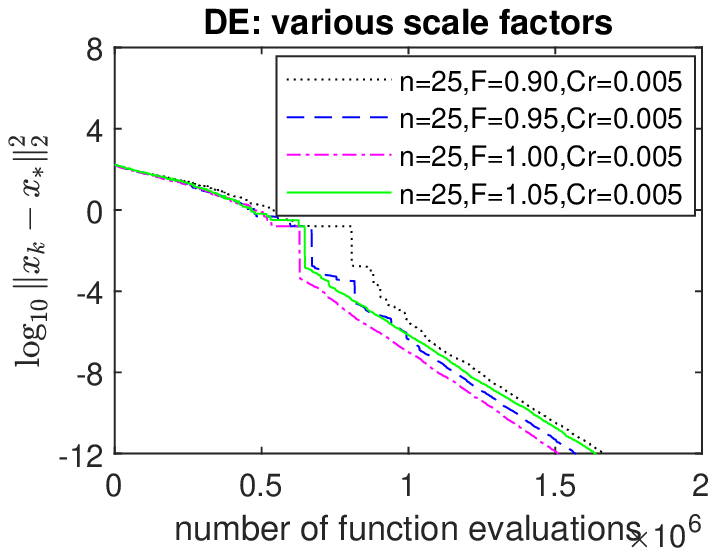}}
\subfigure{\includegraphics[width=0.325\textwidth]{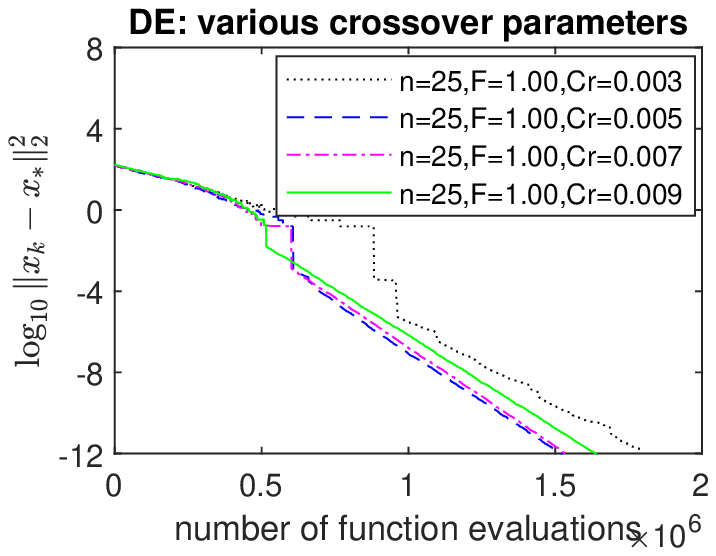}}
\caption{Comparisons of APP and DE for the revised Rastrigin function in various dimensions. For each APP iterative sequence, every initial iterate is randomly selected on a sphere of radius $\sqrt{d}$ centered at the origin, the parameter $\lambda=1/\sqrt{d}$. The search domain of DE is set to a hypercube $[-1,1]^d$. Since the search domain of APP can be viewed as a ball with radius $3\sqrt{d}$, the problems handled by the APP is actually \emph{more difficult} than those handled by the DE. For every comparison shown in the left plots of each row, the parameter values of DE are fine-tuned, as shown in the middle and right plots of each row.}
\label{APP:fig:3}
\end{figure}

\section{Conclusions}
\label{APP:s5}

In this work, we have established an asymptotic representation formula of nonconvex proximal points and have proposed asymptotic proximal point methods for finding the global minima of a class of multiple minima functions. These derivative-free methods have a total work complexity bound $\mathcal{O}(\log\frac{1}{\epsilon})$ to find a point such that the optimality gap between this point and the global minimizer is less than $\epsilon$. Numerical experiments and comparisons demonstrate both the global linear convergence and the logarithmic work complexity of the proposed method.

The algorithm is implemented in Matlab. The source code of the implementation is available at https://github.com/xiaopengluo/app.

Future research is currently being conducted in the following areas. One of the attempts is to establish an adaptive selection strategy for parameters. The empirical choice of parameters depends on a number of comparative experiments, this requires a lot of computational cost. A successful achievement will make our methods more suitable for large-scale applications, meanwhile, it also helps to reduce the length of the inner cycle, i.e., $n$, as much as possible.

Second, we are considering how to extend the assumption of our methods without significantly increasing the computational cost. It is very valuable to efficiently find the best local minima in a certain range for a general nonconvex problem. And our methods increase the possibility of achieving this purpose.

Third, we also hope to investigate further properties of the proposed asymptotic formula. Our work obviously relies on some interesting properties of this formula. It is the key to transform from the differential viewpoint to the integral viewpoint. And further exploration may lead to other ideas for essential nonconvex and nonsmooth optimization problems.

\begin{acknowledgements}
We thank Prof. Jong-Shi Pang for his helpful comment on the previous choice of the parameter $\alpha_k$. 
\end{acknowledgements}


\appendix

\section*{Appendix A}
\label{APP:apd:A}

\begin{proof}[Lemma \ref{APP:lem:2int}]
Let $a^{(i)}$ be the $i$th component of a vector $a\in\mathbb{R}^d$, then for any $1\leqslant i\leqslant d$, when $\alpha(\beta+\gamma)>0$, one obtains
\begin{align*}
  I_1^{(i)}:=&\int_{\mathbb{R}}\exp\left[-\alpha\left(\frac{\beta}{2}
  \big(x^{(i)}-u^{(i)}\big)^2+\frac{\gamma}{2}
  \big(x^{(i)}-v^{(i)}\big)^2\right)\right]\ud x^{(i)} \\
  =&\int_{\mathbb{R}}\exp\left[-\frac{\alpha(\beta+\gamma)}{2}
  \left(x^{(i)}-\frac{\beta u^{(i)}+\gamma v^{(i)}}
  {\beta+\gamma}\right)^2-\frac{\alpha\beta\gamma
  (u^{(i)}-v^{(i)})^2}{2(\beta+\gamma)}\right]\ud x^{(i)} \\
  =&\exp\left(-\frac{\alpha\beta\gamma
  (u^{(i)}-v^{(i)})^2}{2(\beta+\gamma)}\right)
  \int_{\mathbb{R}}\exp\left[-\frac{\alpha(\beta+\gamma)}{2}
  \left(x^{(i)}-\frac{\beta u^{(i)}+\gamma v^{(i)}}
  {\beta+\gamma}\right)^2\right]\ud x^{(i)},
\end{align*}
using the substitution
\begin{equation*}
  t=x^{(i)}-\frac{\beta u^{(i)}+\gamma v^{(i)}}{\beta+\gamma},
\end{equation*}
this yields
\begin{align*}
  I_1^{(i)}=&\exp\left(-\frac{\alpha\beta\gamma
  (u^{(i)}\!-\!v^{(i)})^2}{2(\beta+\gamma)}\right)\int_{\mathbb{R}}
  \exp\left[-\frac{\alpha(\beta+\gamma)}{2}t^2\right]\ud t \\
  =&\exp\left(-\frac{\alpha\beta\gamma
  (u^{(i)}-v^{(i)})^2}{2(\beta+\gamma)}\right)
  \left(\frac{2\pi}{\alpha(\beta+\gamma)}\right)^{\frac{1}{2}};
\end{align*}
and similarly,
\begin{align*}
  I_2^{(i)}:=&\int_{\mathbb{R}}\big(x^{(i)}-u^{(i)}\big)^2
  \exp\left[-\alpha\left(\frac{\beta}{2}
  \big(x^{(i)}-u^{(i)}\big)^2+\frac{\gamma}{2}
  \big(x^{(i)}-v^{(i)}\big)^2\right)\right]\ud x^{(i)} \\
  =&\exp\left(-\frac{\alpha\beta\gamma
  (u^{(i)}\!-\!v^{(i)})^2}{2(\beta+\gamma)}\right)\int_{\mathbb{R}}
  \left(t+\frac{\gamma(v^{(i)}-u^{(i)})}{\beta+\gamma}\right)^2
  \exp\left[-\frac{\alpha(\beta+\gamma)}{2}t^2\right]\ud t \\
  =&\exp\left(-\frac{\alpha\beta\gamma
  (u^{(i)}\!-\!v^{(i)})^2}{2(\beta+\gamma)}\right)\int_{\mathbb{R}}
  \left(t^2+\frac{\gamma^2(v^{(i)}-u^{(i)})^2}{(\beta+\gamma)^2}\right)
  \exp\left[-\frac{\alpha(\beta+\gamma)}{2}t^2\right]\ud t \\
  =&\exp\left(-\frac{\alpha\beta\gamma
  (u^{(i)}-v^{(i)})^2}{2(\beta+\gamma)}\right)
  \left(\frac{2\pi}{\alpha(\beta+\gamma)}\right)^{\frac{1}{2}}
  \left(\frac{1}{\alpha(\beta+\gamma)}+
  \frac{\gamma^2(v^{(i)}-u^{(i)})^2}{(\beta+\gamma)^2}\right).
\end{align*}
Thus, it follows that
\begin{align*}
  \int_{\mathbb{R}^d}\varphi(x)\ud x=\prod_{i=1}^dI_1^{(i)}
  =&\left(\frac{2\pi}{\alpha(\beta+\gamma)}\right)^{\frac{d}{2}}
  \prod_{i=1}^d\exp\left(-\frac{\alpha\beta\gamma
  (u^{(i)}-v^{(i)})^2}{2(\beta+\gamma)}\right) \\
  =&\left(\frac{2\pi}{\alpha(\beta+\gamma)}\right)^{\frac{d}{2}}
  \exp\left(-\frac{\alpha\beta\gamma\sum_{i=1}^d
  (u^{(i)}-v^{(i)})^2}{2(\beta+\gamma)}\right) \\
  =&\left(\frac{2\pi}{\alpha(\beta+\gamma)}\right)^{\frac{d}{2}}
  \exp\left(-\frac{\alpha\beta\gamma\|u-v\|_2^2}{2(\beta+\gamma)}\right);
\end{align*}
and similarly,
\begin{align*}
  \int_{\mathbb{R}^d}\!\|x-u\|_2^2\varphi(x)\ud x
  =&\sum_{i=1}^d\int_{\mathbb{R}^d}\!\big(x^{(i)}-u^{(i)}\big)^2
  \exp\left[-\alpha\left(\frac{\beta}{2}\|x-u\|_2^2+
  \frac{\gamma}{2}\|x-v\|_2^2\right)\right]\ud x \\
  =&\sum_{i=1}^d\left(I_2^{(i)}\prod_{j\neq i}I_1^{(j)}\right) \\
  =&\left(\frac{2\pi}{\alpha(\beta\!+\!\gamma)}\right)^{\frac{d}{2}}
  \!\exp\left(-\frac{\alpha\beta\gamma\|u\!-\!v\|_2^2}
  {2(\beta+\gamma)}\right)\!\left(\frac{d}{\alpha(\beta\!+\!\gamma)}
  +\frac{\gamma^2\|u\!-\!v\|_2^2}{(\beta+\gamma)^2}\right).
\end{align*}
so the proof is complete.\qed
\end{proof}

\bibliographystyle{spmpsci}
\bibliography{MReferences}


\end{document}